\documentclass{article}
\usepackage{amsmath,amsthm}
\usepackage{amssymb}
\usepackage{upgreek}
\usepackage{epsfig}
\usepackage{graphicx}
\usepackage[
  top=0.7in,
  bottom=0.75in,
  left=0.75in,
  right=0.75in
]{geometry}
\usepackage{setspace}
\usepackage{enumerate}
\usepackage{tikz}
\usepackage{tikz-cd}
\usepackage{xcolor}
\usepackage{mathtools}
\usepackage{mathabx} 
\usepackage{stmaryrd}
\usepackage{thmtools}
\usepackage{etoolbox} 
\newbool{defendused}
\usepackage[normalem]{ulem}
\renewcommand{\underline}{\uline}

\usepackage[backref=page]{hyperref}
\hypersetup{
  colorlinks   = true,          
  urlcolor     = blue,          
  linkcolor    = violet,          
  citecolor   = blue             
}

\newcommand{\defend}{\hfill\mbox{$\lozenge$}}

\theoremstyle{definition}
\newtheorem{definition}{Definition}

\newtheorem{theorem}[definition]{Theorem}
\newtheorem{corollary}[definition]{Corollary}
\newtheorem{example}[definition]{Example}

\newtheorem{remark}[definition]{Remark}

 \AtEndEnvironment{definition}{\defend}

\renewcommand{\thmcontinues}[1]{continued}

\newcommand{\A}{{\mathbb{A}}}

\newcommand{\NN} {{\mathbb N}}		
\newcommand{\PP}{\mathbb{P}}         
\newcommand{\QQ} {{\mathbb Q}}		
\newcommand{\RR} {{\mathbb R}}		
\newcommand{\ZZ} {{\mathbb Z}}

\def\cM{\mathcal{M}}

\newcommand{\Mbar}{\overline{\mathsf M}\vphantom{\cM}}

\def\log{\mathsf{log}}

\newcommand{\Spec}{\operatorname{Spec}}

\usepackage{subcaption}

\renewcommand{\log}{{\mathsf {log}}}

\newcommand{\ul}[1]{{\underline{#1}}}

\makeatletter
\newcommand*{\doublerightarrow}[2]{\mathrel{
  \settowidth{\@tempdima}{$\scriptstyle#1$}
  \settowidth{\@tempdimb}{$\scriptstyle#2$}
  \ifdim\@tempdimb>\@tempdima \@tempdima=\@tempdimb\fi
  \mathop{\vcenter{
    \offinterlineskip\ialign{\hbox to\dimexpr\@tempdima+1em{##}\cr
    \rightarrowfill\cr\noalign{\kern.5ex}
    \rightarrowfill\cr}}}\limits^{\!#1}_{\!#2}}}

\title{An invitation to the enumerative geometry of degenerations}
\author{Dhruv Ranganathan}
\date{January 2026}

\begin{document}
\maketitle

\begin{abstract}
This expository article is an introduction to logarithmic Gromov--Witten (GW) theory. 
We discuss how to study the GW theory of a smooth projective variety via simple normal crossings degenerations. 
We survey several approaches to constructing well-behaved, virtually smooth moduli spaces of stable maps to such degenerations. 
Each irreducible component of the special fiber of a degeneration determines a pair consisting of a variety and a normal crossings divisor, and these pairs carry their own logarithmic GW theory. 
We explain how the GW theory of the general fiber can be expressed in terms of the logarithmic GW theory of these pairs. 
Finally, we discuss applications to tautological classes on the moduli space of curves.
\end{abstract}

\setcounter{tocdepth}{2}
\tableofcontents

\section{Introduction}

Let $X$ be a smooth projective manifold. A large part of modern enumerative geometry deals with the intersection theory of moduli spaces $\mathsf M(X)$ attached to $X$. Rich examples include the Hilbert and Quot schemes, moduli spaces of stable bundles and sheaves, and moduli spaces of curves in $X$. Our focus in this note is the intersection theory on the moduli space of maps from nodal curves to $X$, called \ul{Gromov--Witten (GW) theory}. 

In the last decade, thanks to the work of a number of authors, there have been significant advances in how we think about and use degenerations in enumerative geometry. Given a simple normal crossings degeneration $\mathcal X/B$ with general fiber $X$ and simple normal crossings (snc) special fiber $\mathcal X_0$, one goal of this line of inquiry is to extract information about the intersection theory on $\mathsf M(X)$ from analogous moduli spaces associated to the pieces of the special fiber $\mathcal X_0$, or more precisely its \underline{strata} -- the components, the components of the double locus, those of the triple locus, etc. The outcome is the \underline{logarithmic degeneration formula}. At the level of intersection numbers on $\mathsf M(X)$, i.e. Gromov--Witten invariants, it takes the following form:
\[
\begin{tikzpicture}
\node[
  draw,
  rounded corners=8pt,
  inner sep=12pt
] at (0,0)
{$
\text{GW invariants of } X
\;=\;
\displaystyle \sum_{\text{discrete data}}
\;\Bigg(
\displaystyle 
\prod\limits_{\substack{X_i \subset \mathcal X_0\\ \text{a stratum}}}
\underbrace{\text{``log'' GW invariants of } X_i}_{\shortstack{\small counts of curves\\\small with fixed tangency conditions}}
\Bigg)
$};
\end{tikzpicture}
\]
A similar formula holds for cycles. The precise formulas are stated in the main text and proved in~\cite{MR23}. We give an exposition of this formula and the theory that leads to it. 

The insights that have allowed progress have been both theoretical and practical. On the theoretical side, logarithmic geometry and algebraic stacks have played a key role, while on the practical side, tropical geometry has made the resulting theory computable. We will try to explain how both arise in a natural way.

\subsection{A motivating result}

In order to keep a clear geometric motivation, we anchor things with the following result which uses many aspects of these recent advances. Let $\Mbar_g$ denote the moduli space of stable curves of genus $g$ and let $X_5$ be a smooth quintic threefold in $\mathbb P^4$. 

The structures of Gromov--Witten theory produce, for each degree $d$, an element $\mathfrak m_g(X_5,d)$ in the Chow group $\mathsf{CH}_0(\Mbar_g;\mathbb Q)$. The class is related to curves of genus $g$ and degree $d$ on $X_5$. By a naive parameter count, the dimension of the moduli space of such curves is $0$. Morally, the class $\mathfrak m_g(X_5,d)$ is the sum of the classes of these curves. Note that $\mathsf{CH}_0(\Mbar_g;\mathbb Q)$ is expected to be huge for almost all values of $g$. Therefore, the following statement, first conjectured by Pandharipande and proved in~\cite{MR25}, is perhaps surprising. 

\begin{theorem}
For all values of $g$ and $d$, the $0$-cycle $\mathfrak m_g(X_5,d)$ is proportional to the top Chern class of the tangent bundle of $\Mbar_g$ in the Chow group of $0$-cycles on $\Mbar_g$. Equivalently, the class $\mathfrak m_g(X_5,d)$ is proportional to the class of any $0$-dimensional stratum of $\Mbar_g$. 
\end{theorem}

In the final section, we state strengthenings and variants of the above result that can be obtained using the techniques developed here, together with further consequences for GW cycles and the tautological ring of $\Mbar_{g,n}$. These applications are by no means the limit of the usefulness of the methods discussed here. For instance, logarithmic techniques play a central role in the Gross--Siebert mirror symmetry program~\cite{GS26}, as well as in recent work of Chen--Guo--Janda--Ruan on the GW theory of hypersurfaces~\cite{CJR,GJR}. We do not attempt to discuss these developments here, largely due to our own limitations.

\subsection{Three things about GW theory}

Fix a non-negative integer $g$ for the genus and a curve class $\beta\in H_2(X;\ZZ)$. Our main character is the moduli problem $\Mbar_{g,n}(X,\beta)$ for stable maps. A point of this space corresponds to
\[
f \colon (C,p_1,\ldots,p_n)\to X
\]
where $C$ is at worst nodal, has arithmetic genus $g$, the $p_i$ are distinct smooth marked points, and $f_\star[C]$ is equal to $\beta$. Two maps are isomorphic if there is an isomorphism of pointed curves that commutes with the map. The \underline{stability condition} requires that the automorphism group of $[f]$ is finite. More practically, this means that any contracted subcurve, marked at the $p_i$ and the nodes, should be a stable curve. 

The foundations of stable maps and GW theory require significant technical input. We use three key aspects of the theory off the shelf. 

\begin{enumerate}[(i)]
\item \underline{Properness.} The space $\Mbar_{g,n}(X,\beta)$ is representable by a proper Deligne--Mumford stack. 
\item \underline{Virtual structure.} There is a distinguished Chow homology class, the \underline{virtual fundamental class}
\[
[\Mbar_{g,n}(X,\beta)]^{\sf vir}\in \mathsf{CH}_{\sf vdim}(\Mbar_{g,n}(X,\beta);\QQ),
\]
where
\[
{\sf vdim} = (\mathsf{dim}(X)-3)(1-g)-K_X\cdot\beta+n.
\]
\item \underline{Tautological morphisms.} The space admits a tautological source map
\[
\pi\colon \Mbar_{g,n}(X,\beta)\to \Mbar_{g,n},
\]
obtained by forgetting the map and stabilizing the domain curve. It also admits an {\it evaluation map}
\[
{\sf ev}\colon \Mbar_{g,n}(X,\beta)\to X^n,
\]
given by evaluating $f$ at the marked points. 
\end{enumerate}

We avoid unpacking the black box (ii) too much, but a few words are in order. In modern parlance, the stack $\Mbar_{g,n}(X,\beta)$ has a quasi-smooth derived stack structure. In more down-to-earth terms, one should imagine $\Mbar_{g,n}(X,\beta)$ as having a distinguished presentation as an intersection of two smooth spaces inside a third smooth space, where the expected dimension of the intersection is ${\sf vdim}$, but the intersection may not be transverse. In this case, the virtual class is the refined intersection product. In simple cases, such as when $X$ is toric, this is literally true~\cite[Remark~3.2.2]{CFK}. Virtual/derived structures use the deformation theory of the moduli problem to assert that this is always true locally and sufficiently coherently so as to still give a class. See~\cite{BCM20} for a beautiful introduction to the virtual class and~\cite{FP96} for a construction and discussion of the geometry of the space. 

The virtual structure in (ii) uses the smoothness of $X$, but not its properness, while the properness in (i) uses the properness of $X$, but not its smoothness. A proper, singular scheme relevant to our discussion is the special fiber $\mathcal X_0$ of an snc degeneration. 

Gromov--Witten theory is concerned with classes of the form
\[
\pi_\star\left({\sf ev}^\star(\upalpha)\cap [\Mbar_{g,n}(X,\beta)]^{\sf vir} \right)\ \ \textnormal{in} \ \  \mathsf{CH}^\star(\Mbar_{g,n};\QQ), \ \ \upalpha\in \mathsf{CH}^\star(X^n).
\]
When $n$ is equal to $0$, the target $X$ is a quintic threefold $X_5$, and $\beta$ is the degree $d$ class, we obtain a precise definition of the class $\mathfrak m_g(X_5,d)$ from before. 

One can, of course, study the analogous classes in singular cohomology. The two stories have different features. 

When $\upalpha$ is a Chow or Betti cohomology class of degree complementary to the virtual dimension one can push forward to a point to obtain \underline{Gromov--Witten invariants}. The structure of Gromov--Witten invariants is a rich and significant slice of the subject, but even if one is only interested in the invariants, it can be very useful to have cycle-level statements.

\subsection{Deformations, degenerations, semistable reduction} 

A key feature of the virtual class is that it is unchanged, in an appropriate sense, under deformations of $X$. The deformation invariance is quite striking. For example, any smooth and proper rational surface can be deformed to a toric surface, and the enumerative geometry of toric varieties can be approached by a suite of techniques from equivariant cohomology~\cite{GP99}. 

Nevertheless, the requirement that $X$ is smooth is limiting. For example, if $X\subset \mathbb P^r$ is a smooth degree $d$ hypersurface defined by an equation $f_d$, such as our motivating quintic, a family of smooth deformations 
\[
\mathcal X_t :=\mathbb V(f_d+t\cdot g_d)\subset \PP^r, \ \ t\in \mathbb C
\]
can be obtained by varying the coefficients of the defining equation for $X$. However, often, there is no particular value of $t$ for which $\mathcal X_t$ is both smooth and easier than $X$ itself. What is more tempting, at least to some, is to study the family
\[
\mathcal X_{[t_0,t_1]} = \mathbb V(t_0\cdot f_d+t_1\cdot \prod_{i=1}^d \ell_i), \ \ [t_0,t_1] \in \mathbb P^1,
\]
where $\ell_i$ are generic linear forms. Over the point $[0,1]$ one obtains a singular variety whose components are $\mathbb P^{r-1}$. 

Denote the total family $\mathcal X$ and let $\mathcal X_0$ be the fiber over $[0,1]$. For each smooth fiber $\mathcal X_b$, we can consider the moduli space of stable maps to $\mathcal X_b$ and its virtual class. By intersection theory~\cite[Section~20.3]{Ful98}, we can specialize the virtual class to a class:
\[
[\Mbar_{g,n}(\mathcal X_0,\beta)]^{\sf spec-vir}\in {\sf CH}_{\sf vdim}(\Mbar_{g,n}(\mathcal X_0,\beta);\QQ).
\]
The key question is then:

\[
\begin{tikzpicture}
\node[
  draw,
  rounded corners=8pt,
  inner sep=14pt
] 
{
$\displaystyle
\begin{array}{c}
\textnormal{Can the specialized virtual class } 
[\Mbar_{g,n}(\mathcal X_0,\beta)]^{\sf spec\text{-}vir} \\
\textnormal{be described in terms of curves in the components, the double loci, etc?}
\end{array}
$
};
\end{tikzpicture}
\]

As stated, the question has some hidden complexity. The total space of the family $\mathcal X\to \PP^1$ above is singular; this can already be seen when degenerating a quadric surface to two planes. These kinds of singularities are not something we currently understand how to deal with directly, though there is beautiful work in this direction~\cite{BN22,FFR21}. However, by the semistable reduction theorem~\cite{AK00,KKMSD}, by blowups and base change, we can arrange for a family
\[
\mathcal X'\to B'
\]
with the same general fiber, smooth total space, and whose singular fibers are simple normal crossings divisors in the total space. 

\subsection{Outline}

Let us reset the notation. Let $\mathcal X\to B$ be flat and proper map between smooth varieties, with $B$ a curve, and such that all singular fibers comprise smooth hypersurfaces meeting transversely in $\mathcal X$. A cartoon picture is given in Figure~\ref{fig: snc-degeneration}. Let $\mathcal X_0$ be a singular fiber. The goals of this note revolve around two different but closely related geometric situations.

\begin{figure}
\begin{center}
\includegraphics[scale=0.3]{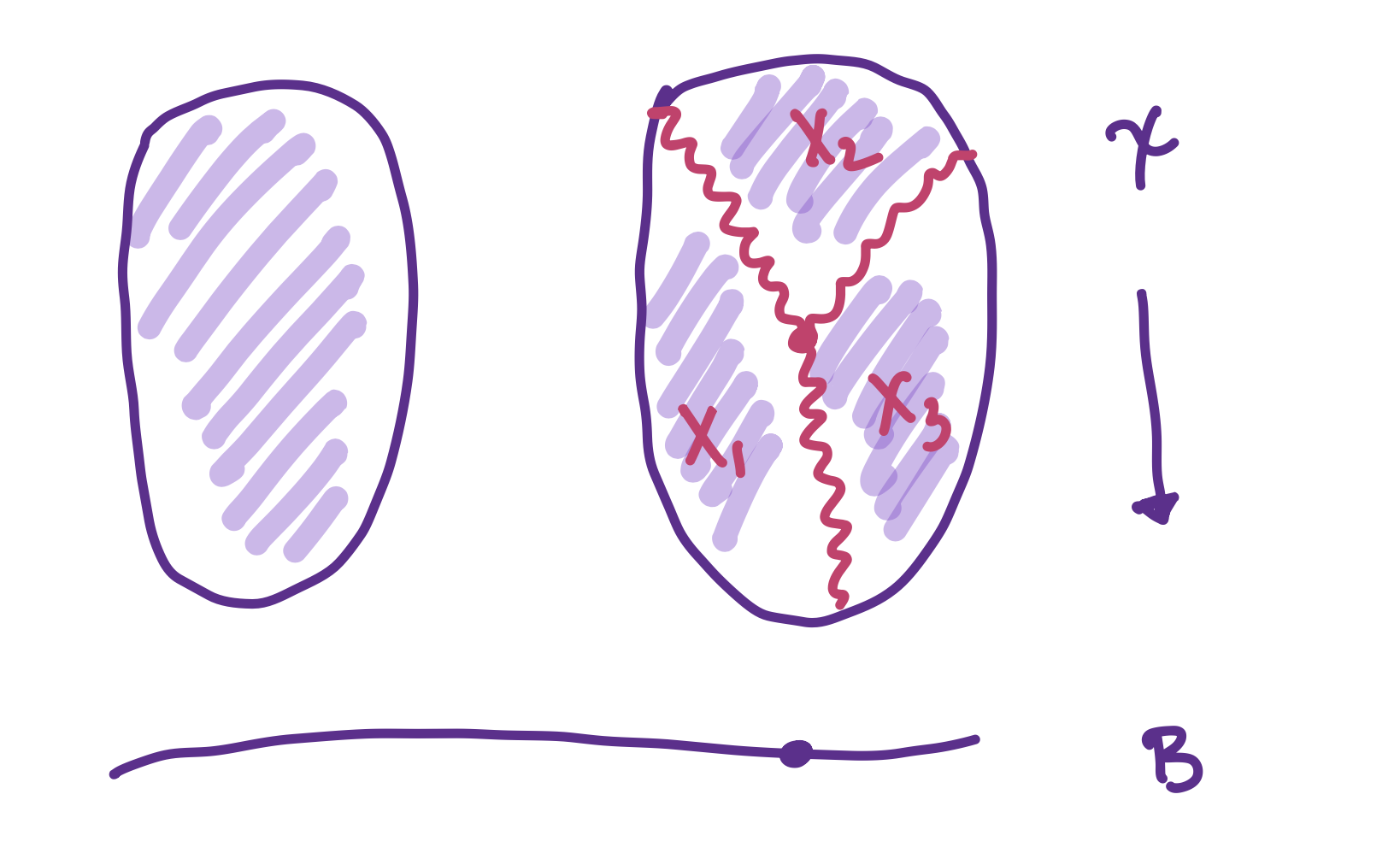}
\end{center}
\caption{A cartoon of a surface degenerating into three components meeting transversely.}\label{fig: snc-degeneration}
\end{figure}

\subsubsection{The degeneration theory}

Our first task is the following.

\[
\begin{tikzpicture}
\node[
  draw,
  rounded corners=8pt,
  inner sep=14pt
]
{
$\displaystyle
\begin{array}{c}
\textnormal{Describe a replacement for the moduli space of stable maps to the singular fiber } \mathcal X_0 \\
\textnormal{that carries an intrinsic virtual class equal to } 
[\Mbar_{g,n}(\mathcal X_0,\beta)]^{\sf spec\text{-}vir}.
\end{array}
$
};
\end{tikzpicture}
\]

We describe three approaches to this problem, all of which lead to equivalent answers. The first is ``cheap'', possibly unusable on its own, but has the feature of feeling ``inevitable'' subject to a few requirements. 

The second is the approach via ``expansions''. This came out of ideas in symplectic geometry~\cite{LR01,IP01}, and has precedents in the study of vector bundles~\cite{Gie84}. Jun Li used these expansions to study the GW theory of ``double point'' degenerations~\cite{Li01,Li02}, and the snc case came later, in work of Maulik and the author~\cite{MR23,R19}. 

The basic idea is to identify a well-behaved but non-compact subspace of the space of stable maps to $\mathcal X_0$ defined by a transversality condition, and compactify it by prioritizing this condition. The approach differs, at least at first sight, from the cheap solution, but gives the same invariants/cycles and is easily amenable to geometric arguments. 

The third is the ``logarithmic geometry'' solution, first envisioned by Siebert, and then carried out by Abramovich--Chen~\cite{AC11,Che10} and Gross--Siebert~\cite{GS13}. The approach here is that when viewed correctly in the category of logarithmic schemes, the map $\mathcal X\to B$ is \underline{smooth} and so $\mathcal X_0$ should be viewed as a smooth object in this category. Again, this differs only cosmetically from the cheap solution, but the logarithmic context is very useful conceptually. 

In both the logarithmic and expanded approach, tropical geometry plays a key role. In fact, this article contains an approach to logarithmic GW theory that completely circumvents logarithmic geometry. However, tropical geometry and the closely related stacks called \underline{Artin fans}, play a prominent role in all the approaches. Tropical geometry also plays a fundamental role in B. Parker's approach~\cite{Par11,Par17a}.

\subsubsection{The strata theory}

Our second task is the following.

\[
\begin{tikzpicture}
\node[
  draw,
  rounded corners=8pt,
  inner sep=14pt
]
{
$\displaystyle
\textnormal{Describe the specialized virtual class in terms of a stable maps theory associated to the strata of } \mathcal X_0.
$
};
\end{tikzpicture}
\]

The strata of $\mathcal X_0$ are smooth varieties, but come equipped with a natural simple normal crossings divisor -- the divisor where they meet deeper strata. The right objects to consider turn out not just to be the deeper strata themselves, but the components obtained by blowing up these strata in $\mathcal X_0$, i.e. projective normal bundles over the strata. Let us write $(X|\partial X)$ for the pair consisting of such a target and its natural boundary divisor. We will explain how the appropriate ``strata'' version of the problem is about maps:
\[
(C,p_1,\ldots, p_n)\to (X|\partial X)
\]
where the $p_i$ have \underline{prescribed tangency} along the components of $\partial X$. The prescription of tangency will act as a ``boundary condition'', and we will glue together maps with matching boundary conditions to describe maps to $\mathcal X_0$. 

In parallel with the degenerate theory, there are three approaches to GW theory in this context of pairs. However, for the purposes of describing the specialized virtual class in terms of stable maps to the strata, the expanded and logarithmic approaches lead to different answers. Again, tropical geometry plays a key role in working with the spaces, regardless of the approach. 

Finally, we put all of these pieces together to state the degeneration formula in Section~\ref{sec: deg-formula}.

\subsubsection{The rest of the document}

In the final few sections of the paper, we collect several basic examples and then outline a few of the applications, including how the results eventually lead to the theorem stated in the introduction. 

\subsection{Further reading}

We emphasize that this article is expository and is based on a number of papers, including~\cite{ACW,AC11,ACGS15,AMW12,AW,CheDegForm,Che10,GS13,Kim08,KLR23,MR20,MR25,MR23,R19}. The main goal of this article is to communicate the basic foundational picture of logarithmic GW theory, at least from the author's perspective. While the exact logical presentation might not appear elsewhere, the ideas may be well-known to experts. There is an obvious intellectual debt to the work of Jun Li~\cite{Li01,Li02}. Another such debt is to a parallel theory, closer to symplectic geometry, which has been put forward in the beautiful work of B.~Parker~\cite{Par17a,Par19,Par11}. 

A reader interested in using this theory and in need of examples where calculations are carried out may consult~\cite{Bou17,KHSUK,MR16,NS06,VGNS}. A reader interested in learning the basics of tropical geometry in an enumerative context can consult Markwig's chapters in the book~\cite{CMRbook}. 

There are many further relevant results and ideas that we will not have time to say much about, but that might interest the reader. A recent development is \underline{punctured} logarithmic maps~\cite{ACGS25} and the closely related refined punctured theory~\cite{BNR24}, which gives a different approach to the degeneration formula. Logarithmic degeneration formulas for Donaldson--Thomas invariants, and the conjectural correspondence with Gromov--Witten theory, are presented in~\cite{MR26,MR20,MR23}. Wise has studied logarithmic Hom schemes in broad generality~\cite{Wis16a} and Kennedy-Hunt has recently developed a very general theory of \underline{logarithmic Hilbert and Quot schemes}, see~\cite{KH23}.

\subsection*{Acknowledgements} 

The text is a distillation of many conversations I have had with my students P. Kennedy-Hunt, E. Dodwell, S. Kannan, S. Koyama, S. Mok, Q. Shafi, T. Song, A. Urundolil Kumaran, and W. Zheng. Given this, I especially hope it will be helpful to young people. I learned this material from many friends and collaborators, and I especially wish to thank D. Abramovich, R. Cavalieri, A. Cela, M. Gross, H. Markwig, D. Maulik, S. Molcho, N. Nabijou, S. Payne, and J. Wise. I am very grateful to D. Abramovich, N. Nabijou, Q. Shafi, and R. Thomas for some extremely valuable feedback on an earlier draft. 

A significant portion of this text was written on British Airways Flight BA035 from London LHR to Chennai MAA and later on the return flight BA036. We thank both flight crews for ideal working conditions. 

The work was supported by EPSRC Horizon Europe Guarantee EP/Y037162/1.

\section{Curves in degenerations}

We keep the notation of the introduction, so $\mathcal X/B$ is an snc degeneration over a curve with special fiber $\mathcal X_0$ and general fiber $\mathcal X_\eta$. Our goal is to build a limiting special fiber for the general fiber $\Mbar_{g,n}(\mathcal X_\eta,\beta)$ that also carries a specialization of the virtual class. We will call this the space of \underline{cheap logarithmic maps} and denote it $\Mbar^{\sf ch}_{g,n}(\mathcal X/B,\beta)$. The basic point is that this is really the smallest space that can define the theory that we want, so we now formally describe what we want. 

\subsection{Two basic requirements}

Two simple requirements pin down $\Mbar^{\sf ch}_{g,n}(\mathcal X/B,\beta)$ completely. \\

\noindent
\underline{Properness requirement.} First, in order to effectively do intersection theory on the space, we would like the space $\Mbar^{\sf ch}_{g,n}(\mathcal X_0,\beta)$ to have a proper map to the stack of ordinary stable maps
\[
\Mbar^{\sf ch}_{g,n}(\mathcal X/B,\beta)\longrightarrow \Mbar_{g,n}(\mathcal X/B,\beta).
\]
We could go a step further and ask for this map to be finite or even a closed immersion. The latter condition would mean we could think of the proposed space of maps as ``ordinary maps satisfying certain properties.'' As we will see, however, allowing the map to be finite or proper provides more flexibility in the modular interpretation. 

In any event, by the valuative criterion for properness, every sequence of stable maps in $\Mbar_{g,n}(\mathcal X_\eta,\beta)$ should have a limit in $\Mbar^{\sf ch}_{g,n}(\mathcal X/B,\beta)$. \\

\noindent
\underline{Base change property.} The next condition is motivated by a property of the usual space of stable maps, namely that the definition of $\Mbar_{g,n}(X,\beta)$ makes essentially no reference to the global geometry of $X$. In practice, this is a key feature when working with the space. 

In the setting of a degeneration, one way to formulate this ``lack of dependence'' on the global geometry is as follows. Let $\mathcal X/B$ and $\mathcal Y/B$ be snc degenerations over the same base $B$. We say a map
\[
\begin{tikzcd}
\mathcal X\ar{rr}\ar{dr}& & \mathcal Y\ar{dl}\\
&B&
\end{tikzcd}
\]
is \underline{strict} if every irreducible component of the special fiber $\mathcal X_0$ is the pullback of an irreducible component of the special fiber $\mathcal Y_0$. 

We impose the following functoriality requirement on the eventual definition of 
$\Mbar^{\sf ch}_{g,n}(\mathcal X/B,\beta)$: whenever we have a strict morphism 
$\mathcal X \to \mathcal Y$ over $B$, the induced square
\[
\begin{tikzcd}
\Mbar^{\sf ch}_{g,n}(\mathcal X/B,\beta)\arrow{r}\arrow{d} 
& \Mbar^{\sf ch}_{g,n}(\mathcal Y/B,\beta)\arrow{d} \\
\Mbar_{g,n}(\mathcal X/B,\beta)\arrow{r} 
& \Mbar_{g,n}(\mathcal Y/B,\beta)
\end{tikzcd}
\]
is Cartesian, where the bottom row consists of the ordinary moduli spaces of maps.

In particular, the cheap space for $\mathcal X$ is obtained from that of $\mathcal Y$ by base change. Thus, the construction incorporates no additional global information from $\mathcal X$ beyond what is already present for $\mathcal Y$. One might then suspect that it depends only on the combinatorics of the special fiber, and not on the ambient geometry.

This requirement leads naturally to the following schematic containment:
\[
\begin{tikzpicture}
\node[
  draw,
  rounded corners=8pt,
  inner sep=12pt
] at (0,0)
{$
\left\{
\text{Smoothable maps to } \mathcal X_0
\right\}
\;\subset\;
\left\{
\begin{array}{c}
\text{Maps to $\mathcal X_0$ that are smoothable after}\\
\text{composing with a \underline{particular} strict } \mathcal X \to \mathcal Y
\end{array}
\right\}
\;\subset\;
\left\{
\text{Cheap log maps}
\right\}
$};
\end{tikzpicture}
\]

The first set is the minimal one required to ensure properness, and must therefore be contained in any reasonable definition of the cheap space. The second set comes from the functoriality requirement: if $\mathcal X \to \mathcal Y$ is strict, then smoothability of a map to $\mathcal X_0$ implies smoothability of the induced map to $\mathcal Y_0$. However, the converse is not always true, so the containment above is usually strict. Requiring the Cartesian property above therefore enlarges the space each time we impose functoriality for a particular strict map. The idea is that the cheap space should include everything that is smoothable after \underline{some} projection along a strict map.

Put differently, suppose
\[
\mathcal X \to \mathcal Y \to \mathcal W
\]
are strict morphisms over $B$. If our cheap space contains maps to $\mathcal X_0$ that become smoothable after projection all the way to $\mathcal W$, we in particular include the maps that are smoothable after projection to $\mathcal Y$.

\subsection{A cheap trick}

By the base change discussion above, we see that if there were a choice of $\mathcal Y/B$ such that every other $\mathcal X$ had a strict map to it, that would completely describe the space. Amazingly, there is such a space and it is very simple to construct.

Recall that the stack $[\mathbb A^1/\mathbb G_m]$ is the moduli stack of pairs $(L,s)$ consisting of a line bundle and a section, and similarly $[\mathbb A^r/\mathbb G_m^r]$ is the moduli stack of $r$-tuples of such data. Given a scheme $S$ and a collection of Cartier divisors $D_1,\ldots, D_r$, there is a moduli map
\[
S\to [\mathbb A^r/\mathbb G_m^r]
\]
such that $D_j$ is the pullback of the divisor where the $j$th coordinate vanishes. 

We have a degeneration $\mathcal X\to B$, and we choose, arbitrarily, a labeling of the components of its special fiber. We therefore obtain a commutative square
\[
\begin{tikzcd}
{\mathcal X} \arrow{d}\arrow{r} & {[\mathbb A^r/\mathbb G_m^r]}\arrow{d}\\
B\arrow{r} & {[\mathbb{A}^1/\mathbb{G}_m]}.
\end{tikzcd}
\]
The right vertical map is induced by taking the complement of the open orbit in $[\mathbb A^r/\mathbb G_m^r]$; this is a Cartier divisor, and so we obtain a moduli map. The bottom arrow is induced by the Cartier divisor $0$ in $B$. Let us define
\[
\mathcal A_B = [\mathbb A^r/\mathbb G_m^r]\times_{[\mathbb A^1/\mathbb G_m]} B.
\]
Let $\mathcal A_{B,0}$ denote the special fiber. Note that $\mathcal A_B\to B$ is relatively $0$-dimensional, and away from the point $0$ it is an isomorphism. The map is stacky over $0$. 

The stack $\mathcal A_B\to B$ has a universal property: for every snc degeneration $\mathcal Y\to B$, together with a labeling of its special fiber components by $\{1,\ldots,r\}$, there is an induced map over $B$
\[
\mathcal Y\to \mathcal A_B. 
\]
Furthermore, this map is strict, in the sense described above. 

We now recall our two requirements: (i) the limit of every stable map to the general fiber of a degeneration should be a cheap logarithmic map, and (ii) the space of logarithmic maps should be well behaved under strict maps. This leads to the following definition of the space of logarithmic stable maps.

\begin{definition}[Cheap logarithmic maps]
Let $\mathfrak M_{g,n}(\mathcal A_B/B)$ denote the stack of maps from prestable curves to the fibers of $\mathcal A_B/B$. Define $\mathfrak M^{\sf ch}_{g,n}(\mathcal A_{B}/B)$ to be the closure of the space of maps to the general fiber
\[
\mathfrak M_{g,n}(\mathcal A^\circ_B/B^\circ)\hookrightarrow \mathfrak M_{g,n}(\mathcal A_B/B).
\]
Define the stack $\Mbar^{\sf ch}_{g,n}(\mathcal X/B,\beta)$ by the pullback
\[
\begin{tikzcd}
\Mbar^{\sf ch}_{g,n}(\mathcal X/B,\beta)\arrow{r}\arrow{d} & \mathfrak M^{\sf ch}_{g,n}(\mathcal A_{B}/B) \arrow{d} \\
\Mbar_{g,n}(\mathcal X/B,\beta)\arrow{r} & \mathfrak M_{g,n}(\mathcal A_B/B). 
\end{tikzcd}
\]
\end{definition}

This definition carries a natural virtual fundamental class. It arises from the deformation theory of the morphism 
\[
\Mbar_{g,n}(\mathcal X/B,\beta)\longrightarrow \mathfrak M_{g,n}(\mathcal A_B/B),
\]
by studying its relative tangent directions. Concretely, the map $\mathcal X\to \mathcal A_B$ is smooth over $B$, and hence admits a relative tangent bundle\footnote{In fact, this is simply the relative logarithmic tangent bundle of $\mathcal X/B$.}. At a stable map $[C\to \mathcal X]$, the relative obstruction theory is given by $R^\bullet \pi_\star T^{\sf log}_{\mathcal X/B}$. This suffices to construct a virtual pullback in the sense of Manolache~\cite{Mano12}.

The stack $\mathfrak M^{\sf ch}_{g,n}(\mathcal A_{B}/B)$ is pure of relative dimension $3g-3+n$ over $B$, and therefore we obtain a class
\[
[\Mbar^{\sf ch}_{g,n}(\mathcal X_0,\beta)]^{\sf vir}
\in \mathsf{CH}_\star(\Mbar^{\sf ch}_{g,n}(\mathcal X_0,\beta);\mathbb Q).
\]
For formal reasons, this virtual class is the specialization of the virtual class of the space of maps to the general fiber.

In summary, because $\mathcal X\to \mathcal A_B$ is smooth over $B$, the same deformation-theoretic arguments used in standard Gromov--Witten theory produce a virtual class in this setting.

As we explain shortly, the definition given here yields the same theory as the other approaches~\cite{AC11,Che10,GS13,R19,MR23}, and presumably also~\cite{Par11}. The definition itself is closest in spirit to work of Gathmann~\cite{Gat02}.  

\medskip
\noindent
\underline{Upshot.} Taken on its own, the definition above is conceptually simple but not especially practical, as it involves a built-in closure operation. Nevertheless, it captures a key conceptual point: this is the minimal theory satisfying the desired properties and is therefore, in a precise sense, inevitable. The more technically involved constructions that follow may be viewed as methods for turning this foundational definition into a workable theory.

\subsection{Expansions} 

The definition of the space $\Mbar^{\sf ch}_{g,n}(\mathcal X/B,\beta)$ involves a closure operation. One modular interpretation—after a certain birational modification—is provided by the study of expanded degenerations. The starting point is the following ``transversalization'' result. 

Let $C_0\to \mathcal Y_0$ be a stable map from a nodal curve to a normal crossings variety. Recall that $\mathcal Y_0$ is naturally stratified into its smooth locus, its locus of double points, triple points, and so on. We say the map is \underline{dimensionally transverse} if the generic point of every irreducible component of $C_0$ maps to the smooth locus of $\mathcal Y_0$, the preimage of the double locus of $\mathcal Y_0$ is a union of nodes of $C_0$ such that at each such node the two branches of $C_0$ map to different irreducible components of $\mathcal Y_0$, and the image of $C_0$ is disjoint from the closure of the triple locus.

We now show that limits of stable maps to the general fiber of a degeneration can always be modified to dimensionally transverse maps. We need some terminology. Let $\mathcal X/B$ be a simple normal crossings degeneration with special point $0\in B$. Given a ramified base change $B'\to B$, the special fiber of the pullback family $\mathcal X_{B'}$ will typically not have simple normal crossings singularities. However, the family is still ``orbifold snc'' or ``toroidal.'' We will be interested in blowups of these pullback families centered in strata in the special fiber, and we call them {\it strata blowups}; we will always base change so that the special fiber is reduced. We call the special fiber of such a blowup an \underline{expanded degeneration}, or simply an \underline{expansion}. See Figure~\ref{fig: exp} for a visualization.
\begin{figure}[h!]
\begin{center}
\includegraphics[scale=0.5]{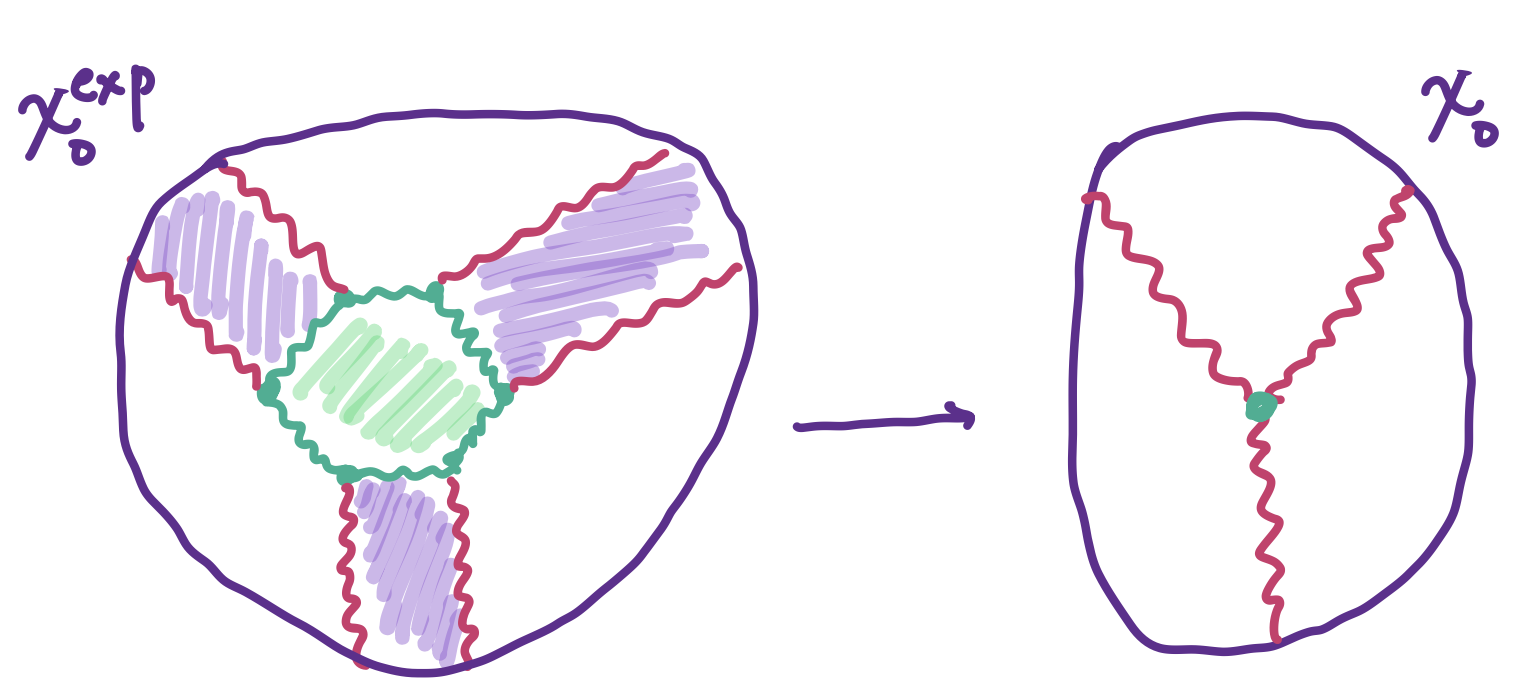}
\end{center}
\caption{Starting with three surfaces meeting at a point, visualized on the right, blowup at the triple point and then each double curve, and base change. The result is an expansion that might be visualized as on the left.}\label{fig: exp} 
\end{figure}

We can now state the theorem. 

\begin{theorem}
Let $\mathcal X \to B$ be a simple normal crossings degeneration with special fiber $X_0$ over $0 \in B$. Let $\Delta$ be the spectrum of a DVR mapping to $B$, with generic point $\eta$ mapping to $B^\times$ and closed point $0_\Delta$ mapping to $0 \in B$. Suppose given a family of stable maps over the generic point $\eta$:
\[
\mathcal C_\eta \longrightarrow \mathcal X_\eta := \mathcal X \times_B \eta.
\]
After a ramified base change $\Delta' \to \Delta$, there exists a strata blowup
\[
\rho : \mathcal X' \to \mathcal X_{\Delta'} := \mathcal X \times_B \Delta'
\] 
and an induced family of stable maps:
\[
\begin{tikzcd}
\mathcal C' \arrow{rr}{f'} \arrow{dr} & & \mathcal X' \arrow{dl}\\
& \Delta' &
\end{tikzcd}
\]
such that the special fiber map
\[
f'_{0_{\Delta'}} : \mathcal C'_{0_{\Delta'}} \longrightarrow \mathcal X'_{0_{\Delta'}}
\]
is dimensionally transverse to the strata of $\mathcal X'_{0_{\Delta'}}$. 
\end{theorem}

With some rearrangement of ideas in the literature, this result is a consequence of~\cite{R19}, but we give a different proof to keep the narrative of the text pointing in the intended direction. 

\begin{proof}
We first note that when $\mathcal X/B$ is a double point degeneration, so that $\mathcal X_0$ has only double point singularities, the theorem was proved by Jun Li~\cite{Li01}. The general case can be reduced to this one by a beautiful theorem of Tevelev, strengthened by Ulirsch~\cite{Tev07,U15}. Let $\mathcal Z\subset \mathcal X$ be the image subvariety of $\mathcal C$. The results of Tevelev and Ulirsch show that we can blow up the total space $\mathcal X^\sim\to \mathcal X$ along strata to ensure that the strict transform 
\[
\mathcal Z^\sim \to \mathcal X^\sim
\]
is dimensionally transverse to the strata; that is, the intersections with strata are either empty or of the expected dimension. We can now delete the locus in $\mathcal X^\sim$ where more than $2$ components of the special fiber meet and call this $\mathcal X^{\sim,\circ}$. Because of dimensional transversality, the subvariety $\mathcal Z^\sim$ factors through this open immersion, so we have
\[
\mathcal Z^\sim\subset \mathcal X^{\sim,\circ}\subset\mathcal X^\sim.
\]
In particular, the space $\mathcal Z^\sim$ is proper over the base. We may now appeal directly to Jun Li's theorem and apply it to the punctured family of stable maps from the general fiber of $\mathcal C$ to the double point degeneration $\mathcal X^{\sim,\circ}$. This produces a further degeneration such that the limiting stable map is transverse in the required sense. Note that Li states the result for projective degenerations, whereas $\mathcal X^{\sim,\circ}$ is only quasi-projective. However, projectivity is used only to guarantee the existence of a limit of a stable map and the properness of $\mathcal Z^\sim$ is sufficient to guarantee this; the procedure of~\cite[Section~3]{Li01} to construct a transverse limit can be followed verbatim.
\end{proof}

\subsection{Predeformable stable maps}

Given a stable map to the general fiber of $\mathcal X$, after a base change and blowup, we have seen that the limiting stable map is dimensionally transverse. An important discovery of Jun Li, however, is that these limits—when dimensionally transverse—always satisfy an additional condition called \underline{predeformability}, see~\cite[Section~2]{Li01}. 

Let $\mathcal X/B$ be an snc degeneration with special fiber $\mathcal X_0$. Suppose we have a stable map
\[
C_0\to \mathcal X_0
\]
that is \underline{dimensionally transverse}, so that, in particular, the preimage of each connected component of the double locus of $\mathcal X_0$ is a union of nodes. Let $q$ be a node of $C_0$. Note that $q$ cannot be a self-node of a component of $C_0$, since such a map would violate dimensional transversality according to our definition. Let $C_1$ and $C_2$ be the two components that meet at $q$, and let $X_1$ and $X_2$ be the components of $\mathcal X_0$ containing the images of $C_1$ and $C_2$, so that $q$ lies in $X_1\cap X_2$. Let $q_1$ and $q_2$ be the two preimages of $q$ under normalization.

We maintain the above notation for the following:

\begin{definition}
A dimensionally transverse stable map $f_0\colon C_0\to\mathcal X_0$ is called \underline{predeformable at $q$} if the tangency $c_i$ of
\[
(C_i,q_i)\to (X_i, X_1\cap X_2)
\]
at $q_i$ along the divisor $X_1\cap X_2$ is independent of $i$; in other words, if $c_1=c_2$. If $[f_0]$ is predeformable at all nodes mapping to the double locus of $\mathcal X_0$, then we say that $[f_0]$ is \underline{predeformable}. 
\end{definition}

See Figure~\ref{fig: dim-trans} below for a cartoon of dimensionally transverse and non-transverse maps.
\begin{figure}[h!]
\begin{center}
\includegraphics[scale=0.4]{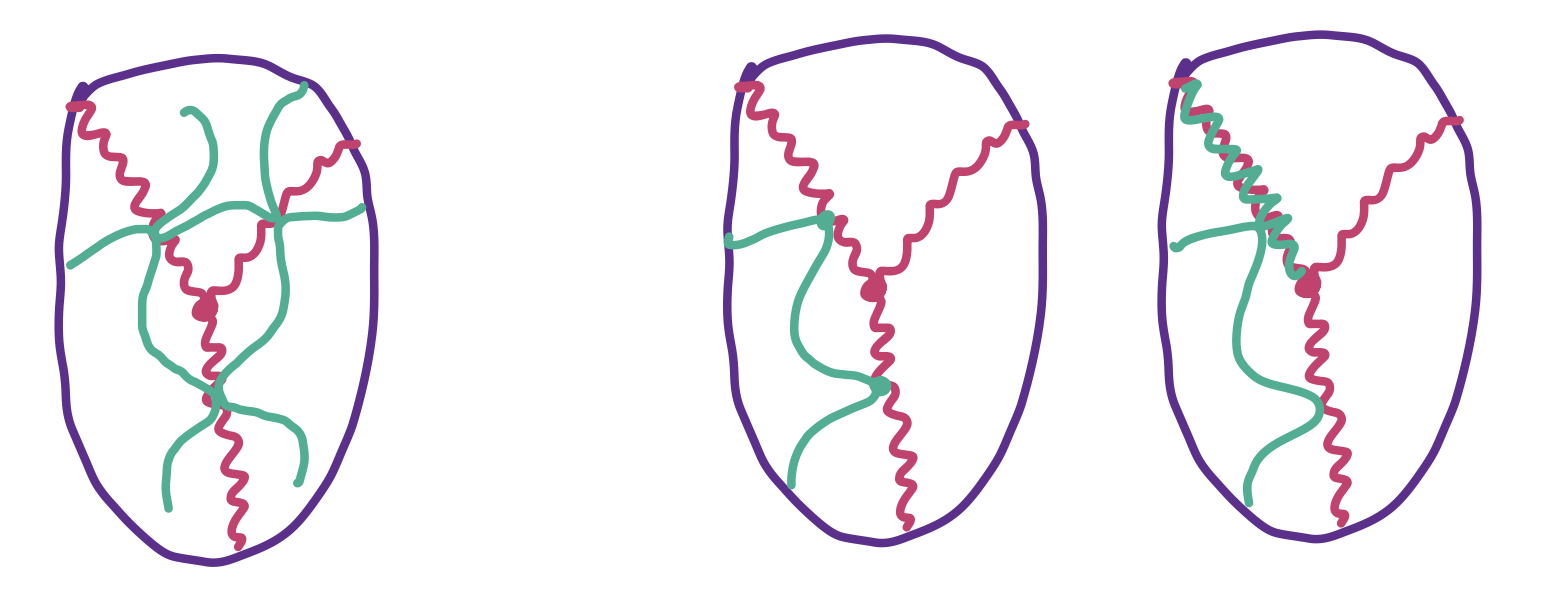}
\end{center}
\caption{The curve on the left is dimensionally transverse, while the two on the right are not.}\label{fig: dim-trans}
\end{figure}

The key point is that at each node of the curve $C_0$, there is a well-defined tangency order of $C_0$ with the double locus of $\mathcal X_0$, possibly $0$ when the node maps to the smooth locus of $\mathcal X_0$. 

As a non-example, we can take $\mathcal X\to B$ to be the standard double point degeneration of a conic in $\mathbb P^2$ given by $XY-tZ^2$. The special fiber $\mathcal X_0$ has two components $X_1$ and $X_2$, both isomorphic to $\mathbb P^1$. We can take $C$ to be $\mathbb P^1$ union $\mathbb P^1$ glued along a single node. Map one component isomorphically onto $X_1$ and the second component by a degree $2$ map ramified over the node $X_1\cap X_2$. This map is dimensionally transverse but fails to be predeformable, see Figure~\ref{fig: non-predef}.
\begin{figure}[h!]
\begin{center}
\includegraphics[scale=0.3]{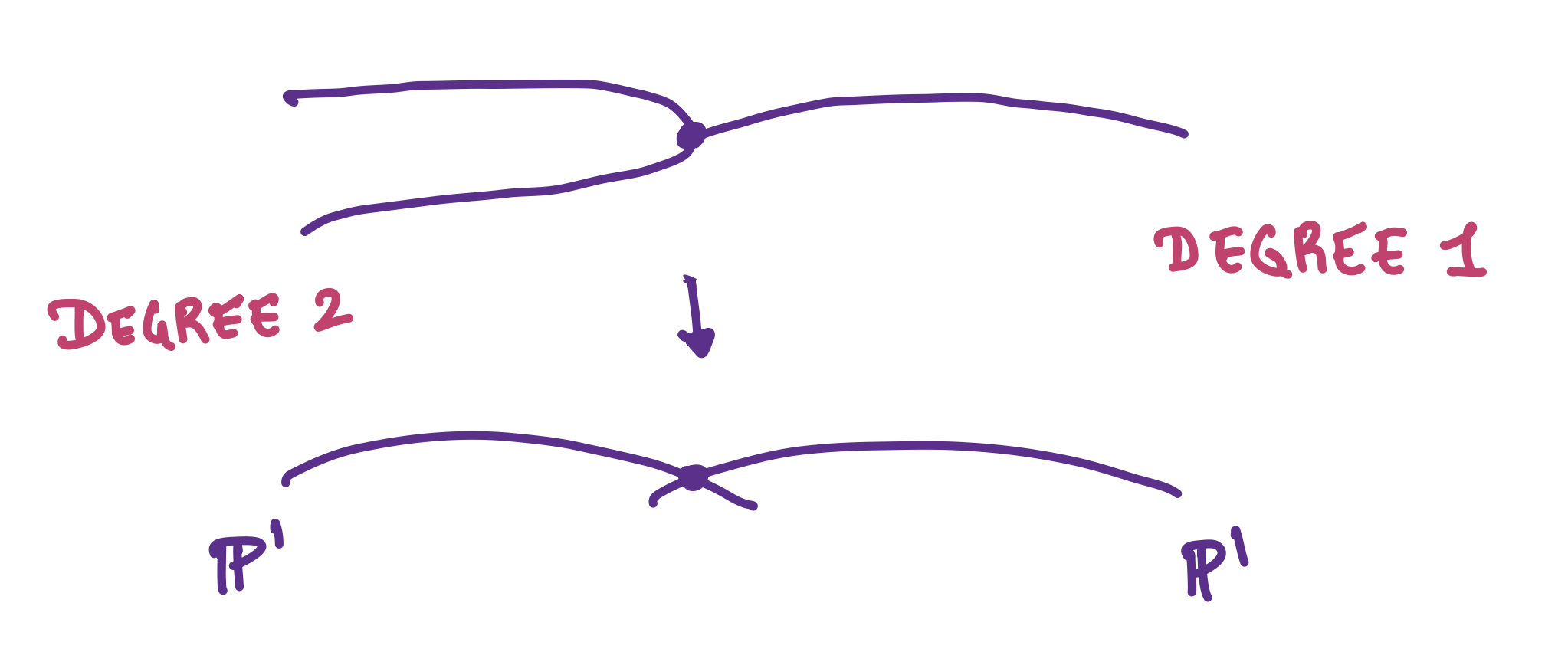}
\end{center}
\caption{A non-predeformable map.}\label{fig: non-predef}
\end{figure}

We flag an important point -- the predeformability condition is local on $\mathcal X_0$ and on the curve. In particular, the condition makes sense on the special fiber $\mathcal A_{B,0}$ of the Artin stack constructed previously. The importance of the locality of the condition is that -- like our ``cheap'' approach, it is essentially independent of the target geometry. The importance of this was already understood by Li~\cite{Li01}.

\begin{theorem}
Suppose $f_0\colon C_0\to \mathcal X_0$ is the stable maps limit of a family of stable maps to the general fiber $\mathcal X_\eta$ of $\mathcal X$. Then if $f_0$ is dimensionally transverse, then it is predeformable. 
\end{theorem}

\begin{proof}
This result is proved by Li~\cite{Li01}, but we give a sketch using toric geometry. Since the statement is local at each node of the curve, we may replace $C_0$ by $\mathbb V(xy)\subset \mathbb A^2_{xy}$. We may also work locally on $\mathcal X$ near the image of the node $q$; in fact, since the statement involves only the orders of vanishing of the components containing $q$, we may replace the target family by
\[
\mathbb A_{uv}^2\to\mathbb A^1_t
\]
where the map is given by $t = uv$. Up to irrelevant smooth directions that we can project away from, this is the local form of a semistable degeneration in a neighborhood of the double locus. 

By hypothesis, we obtain a smoothing of the curve near $q$. By the deformation theory of nodal curves, the smoothing family is given locally by $xy = s^k$, where $k$ is a positive integer. Thus, locally, the smoothing family fits into the diagram:
\[
\begin{tikzcd}
\{xy - s^k = 0\} \arrow{d}\arrow{r} & \mathbb A^2_{uv}\arrow{d}\\
\mathbb A^1_s\arrow{r} &\mathbb A^1_t.
\end{tikzcd}
\]
All the spaces in this diagram are toric. One checks easily that, by hypothesis, under each arrow the preimage of the toric boundary is contained in the toric boundary. Therefore, up to units, we can encode the map via the corresponding map of cones. Note that by the dimensional transversality hypothesis, the two rays in the fan of $\{xy - s^k = 0\}$ map onto the two rays of the fan of $\mathbb A^2$. See Figure~\ref{fig: fan-smoothing}.
\begin{figure}[h!]
\begin{center}
\includegraphics[scale=0.3]{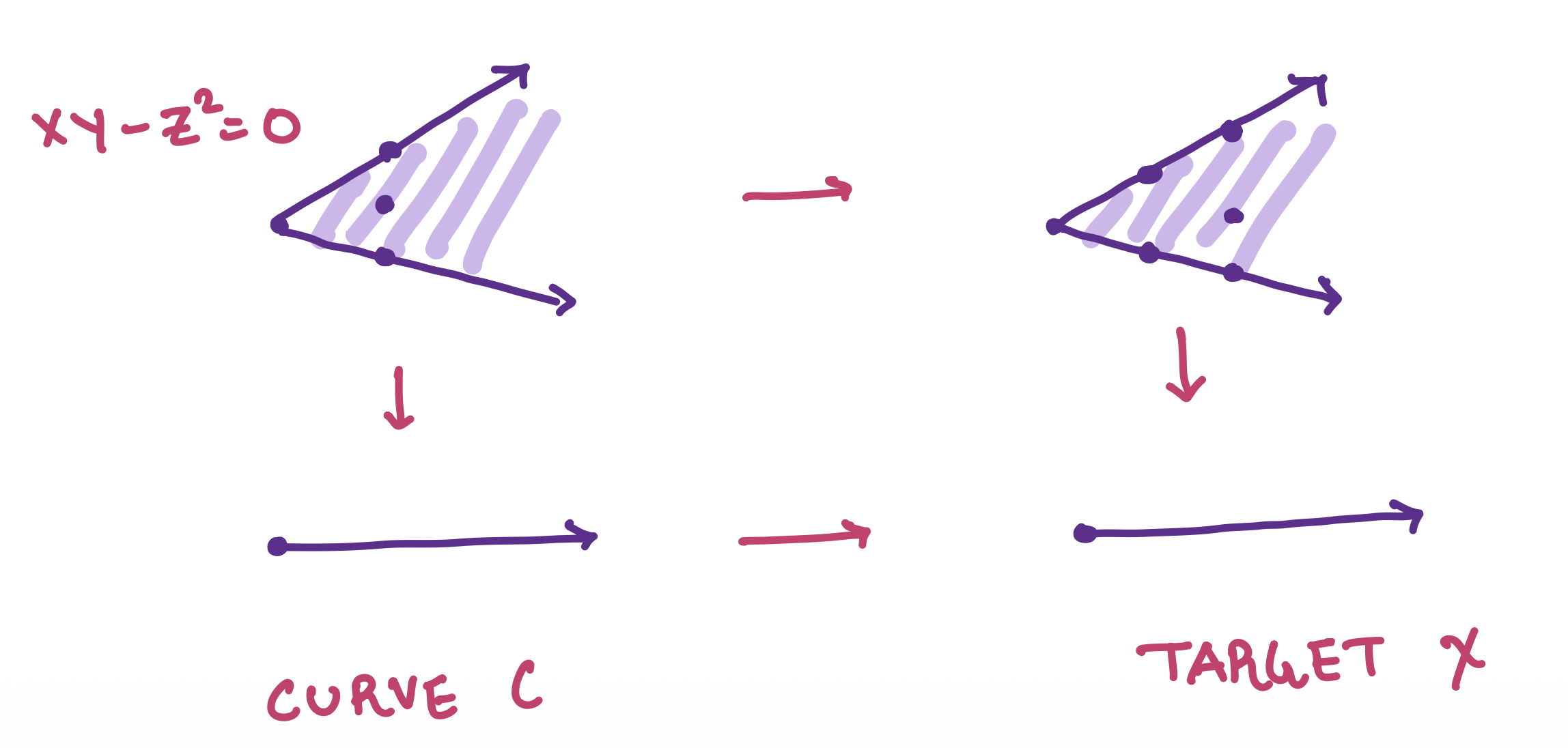}
\end{center}
\caption{The picture at the level of fans of a smoothing predeformable map.}\label{fig: fan-smoothing}
\end{figure}
The predeformability statement now follows by computing the tangency orders using fans. We leave this pleasant exercise to the reader. 
\end{proof}

If $\mathcal X\to B$ is a general degeneration, the converse of the theorem above will typically fail. The predeformability condition is only local, but there can be global obstructions to smoothing. However, we have the following fact, essentially just a perspective on a beautiful observation of Abramovich--Wise~\cite{AW}. 

Recall from our ``cheap trick'' definition of the space of logarithmic maps that we constructed a relatively $0$-dimensional degeneration $\mathcal A_B\to B$ of Artin stacks, built by pulling back a map $[\mathbb A^r/\mathbb G_m^r]\to[\A^1/\mathbb G_m]$ along the tautological map $B\to [\A^1/\mathbb G_m]$. 

\begin{theorem}[Abramovich's land of fairytales]
Let $C_0$ be a nodal curve and $f_0\colon C_0\to \mathcal A_{B,0}$ is a map. Suppose there exists an expansion $\mathcal A'_{B,0}\to \mathcal A_{B,0}$, a destabilization $C_0'\to C_0$ that contracts chains of $2$-pointed rational curves, and a commutative square
\[
\begin{tikzcd}
C_0'\arrow{r}{f_0'}\arrow{d} & \mathcal A'_{B,0}\arrow{d}\\
C_0\arrow{r}{f_0} & \mathcal A_{B,0}.
\end{tikzcd}
\]
such that $f_0'$ is predeformable. Then $[f_0\colon C_0\to \mathcal A_{B,0}]$ is smoothable. Precisely, there is a smoothing of $C_0$ over the spectrum $\Delta$ of a DVR and a family of maps
\[
\begin{tikzcd}
\mathcal C\arrow{d}\arrow{r}{f} & \mathcal A_B\arrow{d}\\
\Delta\arrow{r} & B,
\end{tikzcd}
\]
where the generic point of $\Delta$ maps to the generic point of $B$, such that the special fiber is $[f_0]$.
\end{theorem}

A wonderful thing about this theorem is that the proof is very down to earth. One can just directly construct a smoothing! The basic idea is that there is already a smoothing of $\mathcal A_B'$, and because the map is dimensionally transverse, we can write down a local model for a smoothing of $C_0$ near each node. Since deformation theory of curves is unobstructed, this gives a global smoothing of $C_0$. As we have already seen, to build a map to $\mathcal A_B$ all we need to do is specify some Cartier divisors, and the dimensional transversality means this is straightforward. Now the details. 

\begin{proof}
Let us first assume that $\mathcal A_{B,0}'$ coincides with $\mathcal A_{B,0}$. The general case is similar and we comment on it below. Let $q$ be a node of $C_0$ that maps to the singular locus. The morphism to $\mathcal A_{B,0}$ has a well-defined tangency at $q$, obtained by normalizing and measuring the tangency of the node with the double locus; this is well defined by predeformability. Denote this integer by $m_q$. By elementary deformation theory of nodal curves, there is a smoothing of $C_0$ to a curve
\[
\mathcal C\to \Delta = \Spec R
\]
such that the local model at a node $q$ as above is given by $xy = t^{1/m_q}$, where $R\supset \mathbb C\llbracket t\rrbracket$ is a finite extension containing all the necessary roots. One may choose any such extension; for concreteness, we take $R = \mathbb C\llbracket t^{1/\ell}\rrbracket$, where $\ell$ is the least common multiple of the $m_q$. For the nodes that do not map to the singular locus, we choose the smoothing arbitrarily. In fact, these nodes do not affect the remainder of the argument, so we ignore them. Now choose the map $\Delta\to B$ so that the function cutting out $0\in B$ pulls back to $t$ in $R$; note that $t$ is not a uniformizer in $R$, since $R$ contains fractional powers of $t$. 

We must now construct a map
\[
\mathcal C\to \mathcal A_B.
\]
Since $\mathcal A_B = [\mathbb A^r/\mathbb G_m^r]\times_{[\mathbb A^1/\mathbb G_m]} B$, it suffices to construct a map to $[\mathbb A^r/\mathbb G_m^r]$ commuting with $\mathcal C\to B\to[\mathbb A^1/\mathbb G_m]$. In other words, for each coordinate divisor of $[\mathbb A^r/\mathbb G_m^r]$, we must describe the Cartier divisor on $\mathcal C$ to which it pulls back. Geometrically, for the special fiber of the map to agree with $[f_0]$, equidimensionality implies the following: the pullback of a component of $\mathcal A_{B,0}$ to $\mathcal C$ is a positive linear combination of components of $C_0$ whose generic points map to that component. It remains to determine the multiplicities. This is straightforward: at the generic point of a component of $\mathcal A_{B,0}$, that component is the scheme-theoretic preimage of $0$ in $B$. The function cutting out $0$ pulls back to $t$, and $t^{1/\ell}$ vanishes along the corresponding component of $C_0$ with multiplicity $1$. Note that the total space of the smoothing $\mathcal C$ may have quotient singularities. However, for instance by toric geometry, these divisors can be seen to be Cartier on $\mathcal C$ and so determine a map to $\mathcal A_B$. By construction it agrees with $[f_0]$. 

When the expansion $\mathcal A'_{B,0}\to \mathcal A_{B,0}$ is nontrivial, we can proceed similarly. The same argument shows that $f'_0$ has a smoothing with domain $\mathcal C'$. This implies that $f_0$ has a smoothing by stabilizing the map to $\mathcal C\to\mathcal A_B$. Indeed, this is just the fact that $\mathcal C'$ is a destabilization of some deformation $\mathcal C$ of $C_0$. 
\end{proof}

Recall now that we defined the stack of stable maps by a pullback diagram:
\[
\begin{tikzcd}
\Mbar^{\sf ch}_{g,n}(\mathcal X/B,\beta)\arrow{r}\arrow{d} & \mathfrak M^{\sf ch}_{g,n}(\mathcal A_{B}/B) \arrow{d} \\
\Mbar_{g,n}(\mathcal X/B,\beta)\arrow{r} & \mathfrak M_{g,n}(\mathcal A_B/B)
\end{tikzcd}
\]
where $\mathfrak M^{\sf ch}_{g,n}(\mathcal A_{B}/B)$ is exactly the set of stable maps to $\mathcal A_{B,0}$ that are smoothable. But by the theorem above, this equivalently means that (after destabilization) the map factorizes as a predeformable map through an expansion. Furthermore, given such an expansion, we can pull back along the smooth map $\mathcal X_0\to \mathcal A_{B,0}$ to get an expansion, so we have
\[
\begin{tikzcd}
C_0' 
  \arrow{dr} 
  \arrow[bend left=20]{rrd}
  \arrow[bend right=20]{rdd} 
& &  \\
& \mathcal X'_0 \arrow{d} \arrow{r} 
& \mathcal X_0 \arrow{d} \\
& \mathcal A_{B,0}' \arrow{r} 
& \mathcal A_{B,0}
\end{tikzcd}
\]
The lower of the curved arrows is provided by hypothesis -- because the map is smoothable, there exists a factorization through some expansion. The upper curved arrow is given of course. The map to the expansion $\mathcal X_0'$ is predeformable -- indeed the condition is local and descends along smooth maps. Putting all this together, we have:

\begin{theorem}[Log maps factor through expansions]
The space $\Mbar^{\sf ch}_{g,n}(\mathcal X/B,\beta)$ is the closed substack of the mapping stack $\Mbar_{g,n}(\mathcal X_0,\beta)$ comprising maps $C_0\to \mathcal X_0$ that, after a possible destabilization $C_0'\to C_0$, factorize through a predeformable map 
\[
C_0'\to \mathcal X_0'\to \mathcal X_0
\] 
with the first map predeformable. 
\end{theorem}

Note that this factorization property determines the stack structure on $\Mbar^{\sf ch}_{g,n}(\mathcal X/B,\beta)$. Indeed, the factorization property equivalently picks out $\mathfrak M_{g,n}^{\sf ch}(\mathcal A_B/B)$ inside the stack of maps $\mathfrak M_{g,n}(\mathcal A_B/B)$, and it is reduced. 

\subsection{The stack of expansions} 

The discussion above shows that each point of $\Mbar^{\sf ch}_{g,n}(\mathcal X/B,\beta)$ factors through an expansion. One would like a {\it family} of expansions
\[
\mathcal X^{\sf exp}\to \Mbar^{\sf ch}_{g,n}(\mathcal X/B,\beta)\to B
\]
through which all maps factor. By doing some simple examples, one can observe that the fibers of such a proposed family can fail to vary flatly over the base. However, after a suitable birational modification of $\Mbar^{\sf ch}_{g,n}(\mathcal X/B,\beta)$, such a family does exist. In fact, this modification can be constructed directly, without reference to the space above; this was Jun Li's approach in the double point case. While there are additional technical challenges, Li's method extends to the general setting~\cite{R19}. We briefly summarize this approach.

The key construction is that of a ``stack of expansions''. Let $\mathcal X \to B$ be an snc degeneration with smooth total space. In~\cite{MR20,MR23}, a stack 
$\mathsf{Exp}_{\mathcal X/B}$ over $B$ is constructed. It parametrizes certain expanded 
degenerations of $X/B$ in the sense we have described above -- obtained by base changes and strata blowups. For $f\colon S\to B$, an object of $\mathsf{Exp}_{\mathcal X/B}(S)$ 
is a flat family
\[
\mathcal X^{\sf exp} \to S
\]
and a diagram
\[
\begin{tikzcd}
\mathcal X^{\mathsf{exp}} \arrow[r] \arrow[d] & \mathcal X \arrow[d] \\
S \arrow[r] & B
\end{tikzcd}
\]
such that for every geometric point $s \in S$, the fiber 
$\mathcal X^{\mathsf{exp}}_s$ is obtained from the fiber $\mathcal X_{f(s)}$ as follows:

\begin{itemize}
\item If $\mathcal X_{f(s)}$ is smooth, then 
$\mathcal X^{\mathsf{exp}}_s \cong \mathcal X_{f(s)}$.

\item If $\mathcal X_{f(s)}$ is singular, i.e. if $f(s) = 0\in B$, then 
$X^{\mathsf{exp}}_s$ is an expansion of $\mathcal X_{f(s)}$ 
in the sense described above.
\end{itemize}

The $\mathsf{Exp}_{\mathcal X/B} \to B$ is relatively $0$-dimensional. Informally, we view the stack $\mathsf{Exp}_{\mathcal X/B}$ as the moduli stack of all expansions of $\mathcal X\to B$. However, there are two important features that play a prominent role in the snc story that do not appear in Jun Li's original theory.

\begin{itemize}
\item The stack $\mathsf{Exp}_{\mathcal X/B} \to B$ is not unique, but rather, requires an auxiliary polyhedral choice -- it is a choice of rational polyhedral complex structure on a topological space. The choice does not affect the theory (meaning the cycles, invariants, etc) in a meaningful way, but it does produce a different space. 
\item As a by-product of the choice, given a point of $\mathsf{Exp}_{\mathcal X/B}$, the expansion $\mathcal X^{\sf exp}$ over it carries an additional decoration. Precisely, certain $\mathbb P^1$-bundle components are designated as \underline{tube components}. When studying the stable map moduli space, an additional stability condition must be forced at these components. 
\end{itemize}

Let us explain the last point in slightly different terms. If $\mathcal X_0^{\sf exp}\to \mathcal X_0$ is an expansion, we say that a component is a \underline{$\mathbb P^1$-bundle component} if there is a contraction
$$
\mathcal X_0^{\sf exp}\to \overline{\mathcal X}_0^{\sf exp}
$$
to another expansion that exhibits this component as a $\mathbb P^1$-bundle over a stratum of $\overline{\mathcal X}_0^{\sf exp}$.

A point of the stack of expansions consists of an expansion together with a designation of \underline{some} subset of these $\mathbb P^1$-bundle components as \underline{tubes}. In other words, over every geometric point of the stack of expansions, we have not only an expansion $\mathcal X_0^{\sf exp}$ but also a contraction of the tubes to another expansion
\[
\mathcal X_0^{\sf exp}\to \overline{\mathcal X}_0^{\sf exp}.
\]
Again, this contracts only \underline{some} $\mathbb P^1$-bundle components, not all of them---only those predesignated as tubes.

\begin{example}[Jun Li's stack of expansions]
Suppose $\mathcal X\to B$ is a degeneration with $\mathcal X_0$ consisting of two components $Y_1$ and $Y_2$ meeting transversely at $D$. Then we can express the stack $\mathsf{Exp}_{\mathcal X/B} \to B$  very explicitly. Away from $0$ it is just a point, since there is no way to expand a smooth fiber of $\mathcal X/B$. We describe the special fiber by the objects it parameterizes as a set, then with a topology, and finally its isotropy groups. This does not fully describe the stack of course but should give the reader a feel for what is going on. 

First, an expansion in this geometry is determined by its {\it length} $k \ge 0$, and the length $k$ expansion is
\[
X_0[k]
=
Y_1 \;\cup\; P_1 \;\cup\; \cdots \;\cup\; P_k \;\cup\; Y_2,
\]
where each
\[
P_i = \mathbb{P}\big(\mathcal{O}_D \oplus N_{D/Y_1}\big)
\]
is a $\mathbb{P}^1$-bundle over $D$, and the components are glued in a chain
along their zero and infinity sections.

Now, over the special point $0 \in B$, the stack $\mathsf{Exp}_{X/B}$ has one
isomorphism class of objects for each integer $k \ge 0$, corresponding to the
expansion $X_0[k]$.  

\[
\mathsf{Exp}_{X/B,0}
=
\left\{[0],[1],[2],\dots\right \},
\]
where $[k]$ is the expansion with $k$ inserted $\mathbb{P}^1$-bundles.

We can describe the topology on the stack by describing specializations between points. The specialization order is given by increasing length: the closure of $[k]$
contains all $[m]$ with $m \ge k$.  In other words, longer expansions are more
degenerate.  Altogether, the special fiber of $\mathsf{Exp}_{X/B}$ is a countable chain of strata
\[
[0] \supset \overline{[1]} \supset \overline{[2]} \supset \cdots.
\]

The automorphism group of the length-$k$ expansion is $\mathrm{Aut}(X_0[k]) \cong \mathbb{G}_m^k$. This comes from the independent fiberwise scalings of each bundle $P_i$
which fix the two sections.  Ultimately, the point $[k]$ is a stacky
point $\mathcal B\mathbb{G}_m^k$.

Note that the tubes do not appear in Jun Li's theory. They can be imposed artificially though, e.g. by blowing up the universal family $\mathsf{Exp}_{X/B}$ and semistabilizing. This is explained in detail in~\cite[Section~6.7]{MR20}.
\end{example}

More general stacks of expansions are rarely so well-behaved or explicitly understood. One well-understood case comes from classical toric geometry: the combinatorics of the \underline{secondary fan}~\cite{GKZ} provides a natural model for the stack of expansions of surfaces~\cite{KH21,MR20}. For a parallel study in the setting of configuration spaces, where things are similarly explicit, see~\cite{SCM24}.

\subsubsection{The stack of expansions in general}
In the next example, we describe the stack of expansions in the general case. The basic input comes from toric/toroidal geometry. 

We take $B = \mathbb A^1$ and $\mathcal X = \mathcal A_B$, the degeneration with $r$ components described previously:
\[
\mathcal A_B = [\mathbb A^r/\mathbb G_m^r]\times_{[\mathbb A^1/\mathbb G_m]} \mathbb A^1.
\]
It suffices to construct the stack in this case. Any expansion of $\mathcal A_B$ is equivalently a torus-equivariant degeneration of $\mathbb A^r$ over $\mathbb A^1$. Standard toric geometry, as described in~\cite{KKMSD}, encodes such degenerations combinatorially: they correspond to rational polyhedral decompositions of the standard $r$-simplex $\Delta_r$. Specifically, given a decomposition $\mathcal P$ of $\Delta_r$, taking the cone over $\mathcal P$ yields a fan refining that of $\mathbb A^r$. Passing back to quotient stacks, we obtain a birational modification
\[
\mathcal A_{B}(\mathcal P)\to \mathcal A_B,
\]
and thus a new degeneration over $B$. The special fiber may not be reduced; it is reduced precisely when the vertices of $\mathcal P$ have integer coordinates. Since we start with the standard simplex $\Delta_r$ at height $1$ under $\mathbb R^r_{\geq 0}\to \RR_{\geq 0}$, this only occurs for the trivial subdivision. But if we start with a polyhedral complex with \underline{rational} coordinates subdividing the height $1$ slice, we can dilate the picture, for instance by the least common multiple of all denominators, corresponding to a ramified base change $\mathbb A^1\to \mathbb A^1$. 

The special fiber of the base change is an expansion. Since we only care about codimension-one strata, we focus on the $1$-dimensional skeleton of $\mathcal P$. Such an object—essentially an embedded graph in $\Delta_r$—is called a \underline{$1$-complex}.

Different $1$-complexes can give the same expansion -- it happens exactly when the ``combinatorial'' picture of the $1$-complexes are the same. Each vertex of $\mathcal P$ determines a toric variety, and these toric varieties can be glued to describe the special fiber. When this occurs, we say the $1$-complexes have the same \underline{combinatorial type}. The set of combinatorial types is discrete. In the case $r=2$, i.e., in the double point theory, it reduces to the length of the expansion.

\underline{The break with the double point theory.} We would like to construct the stack $\mathsf{Exp}_{\mathcal X/B}$ whose points are in bijection with expansions, determined by combinatorial types of polyhedral complexes. This is exactly where things begin to diverge from Jun Li's theory of expanded degenerations.  

There is a general procedure, known as the \underline{Artin fan construction}~\cite{AW,ACMW,CCUW}, which takes a polyhedral complex (or a stacky version) and produces a $0$-dimensional Artin stack. The idea to construct $\mathsf{Exp}_{\mathcal X/B}$ is then:
\begin{enumerate}[(i)]
\item Construct a moduli functor on the category of cone/polyhedral complexes, parameterizing families like $\mathcal P$. 
\item Show that this functor is representable by a polyhedral complex (or stacky version). 
\item Apply the Artin fan construction to convert it into an Artin stack over $B$. Functoriality ensures a universal family whose fibers are the required expansions. 
\end{enumerate}

This procedure recovers the Jun Li stack above, giving an efficient construction. In the general case, however, property (ii) fails: the functor is not representable\footnote{Such non-representable functors on the category of polyhedral/cone complexes are not especially exotic. Here is a simple example: consider the functor on the category of rational polyhedral cone complexes that sends $\Sigma$ to the set (in fact, group) of $\mathbb R$-valued piecewise linear functions on $\Sigma$. This clearly should be represented by the object $\mathbb R$ -- indeed this is essentially a tautology. But this is not a cone complex. Indeed, what cone complex structure should it have? E.g. if we give $\mathbb R$ the fan structure of $\mathbb P^1$, morphisms from $\Sigma$ would not capture all such functions, since it would require cones to be mapped to cones. Therefore, the functor wants to be represented by the real line with ``no'' cone complex structure, or perhaps a single non-convex cone given by $\mathbb R$ itself.} Attempting to represent it instead produces a ``polyhedral space''—a topological space underlying a polyhedral complex but without a \underline{distinguished choice} of polyhedral structure. 

The key object in~\cite{MR20} is a topological space $T$ with an infinite family of admissible polyhedral complex structures, together with a bijection
\[
\{\textnormal{Admissible polyhedral structures $\lambda$ on $T$}\} \longleftrightarrow \{\textnormal{Stacks of expansions $\mathsf{Exp}_{\mathcal X/B}^\lambda$}\}.
\] 
These structures form a category under refinement, and are closed under refinement. One can view this as analogous to the dictionary between toric compactifications of a torus and cone complex structures on its cocharacter lattice. 

Each choice of stack of expansions comes with a universal family
\[
\mathcal X^{\lambda,\mathsf{exp}}\to \mathsf{Exp}_{\mathcal X/B}^\lambda
\]
and, as we explained before, a distinguished set of components called \underline{tube components}. The choice of the stack of expansions and the designation of tube components are both prices for attempting to represent a non-representable functor. 

As we will see in the statement of the degeneration formula in Theorem~\ref{thm: deg-form}, the flexibility of building the entire system of these stacks has some utility, though it may seem like a nuisance at first. Even in the double point case, the picture of many stacks of expansions persists; however, the category of admissible polyhedral structures has a terminal object, which recovers the representing functor. Again, we refer to~\cite[Section~6.7]{MR20} for a discussion of exactly how things collapse in this case.

\subsection{The stack of expanded maps}

By choosing a stack of expansions and studying stable maps to the universal family, we obtain a moduli stack of stable maps to expansions. The role of the tube components is to ensure that the universal family varies flatly over the stack of expansions.

Fix a projective snc degeneration $\mathcal X \to B$ and a stack of expansions $\mathsf{Exp}_{\mathcal X/B}$.

\begin{definition}[Expanded stable maps]
The moduli problem of expanded stable maps $\Mbar^{\mathrm{exp}}(\mathcal X/B,\beta)$ has fiber over $S$ given by
\[
(\mathcal X^{\mathrm{exp}}_S \to S,\; C \to S,\; f : C \to \mathcal X^{\mathrm{exp}}_S)
\]
where:

\begin{enumerate}[(i)]
\item $\mathcal X^{\mathrm{exp}}_S \to S$ is an $S$-point of 
$\mathsf{Exp}_{\mathcal X/B}$, i.e.\ an expansion of $\mathcal X/B$ over $S$;

\item $C \to S$ is a family of nodal curves; and

\item $f : C \to \mathcal X^{\mathrm{exp}}_S$ is an $S$-morphism.
\end{enumerate}

These data are required to satisfy:

\begin{enumerate}[(i)]

\item \underline{Predeformability.} At every geometric point of $S$, the map is predeformable. 

\item \underline{Tube condition.} For every geometric point $s$, 
let $T \subset \mathcal X^{\mathrm{exp}}_s$ be a tube component with projection
\[
\pi_T : T \to D
\]
exhibiting it as a $\mathbb P^1$-bundle. If $C' \subset C_s$ is an irreducible component such that 
$f_s(C') \subset T$, then
\[
f_s|_{C'} : C' \longrightarrow T
\]
factors through a fiber of $\pi_T$ and is a finite cover of that 
$\mathbb{P}^1$-fiber, ramified only over the sections 
$0$ and $\infty$.

\noindent
(Tube contraction.) In particular, for every geometric point $s$ of $S$, 
by contracting the tube components and the components of their preimages, 
we obtain a factorization $\overline C_s \to \overline{\mathcal X}_s$. 

\item \underline{Stability.} For every geometric point $s$, 
the map obtained after tube contraction
\[
\overline f_s : \overline C_s \to \overline{\mathcal X}^{\mathrm{exp}}_s
\]
has finite automorphism group over $\mathcal X$. 
\end{enumerate}

The curve class $\beta$ is measured after pushforward to $\mathcal X$. 
\end{definition}

In Jun Li's double point theory, the stability condition requires that, in addition to the usual notion of stability for contracted components in Kontsevich's mapping spaces, each $\mathbb P^1$-bundle component contain at least one component of the curve that maps as something \underline{other} than a fully ramified cover of a fiber. In this version of the theory, we \underline{allow} such ``trivial bubble'' mappings, but only when they map into the tube components.

We fix a projective simple normal crossings degeneration $\mathcal X/B$ and a stack of expansions, and suppress this choice from here on out. We have the following theorem.

\begin{theorem}
The moduli stack $\Mbar_{g,n}^{\sf exp}(\mathcal X/B,\beta)$ of stable maps to expansions is a proper Deligne--Mumford stack over $B$. The morphism
\[
\Mbar_{g,n}^{\sf exp}(\mathcal X/B,\beta)\to \mathsf{Exp}_{\mathcal X/B}
\]
carries a relative perfect obstruction theory. For each $b\in B$, this yields a virtual class 
\[
[\Mbar_{g,n}^{\sf exp}(\mathcal X_b,\beta)]^{\sf vir}.
\]
The natural morphism
\[
\Mbar_{g,n}^{\sf exp}(\mathcal X/B,\beta)\to \Mbar_{g,n}^{\sf ch}(\mathcal X/B,\beta)
\]
is proper, and the virtual classes agree under pushforward.
\end{theorem}

\begin{remark}
The map above is {not} an isomorphism; it has positive-dimensional fibers. If we ignore the stability condition and work with the universal degeneration $\mathcal A_{B}\to B$, we can define an analogous stack $\mathfrak M^{\sf exp}_{g,n}(\mathcal A_{B}/B)$ of maps to expansions of $\mathcal A_{B}$, with stability imposed \underline{relative} to $\mathcal A_{B}/B$. This can be arranged so that the resulting stacks fit into a Cartesian diagram
\[
\begin{tikzcd}
\Mbar^{\sf exp}_{g,n}(\mathcal X/B,\beta)\arrow{r}\arrow{d} & \mathfrak M^{\sf exp}_{g,n}(\mathcal A_{B}/B) \arrow{d} \\
\Mbar_{g,n}^{\sf ch}(\mathcal X/B,\beta)\arrow{r} & \mathfrak M_{g,n}^{\sf ch}(\mathcal A_B/B).
\end{tikzcd}
\]
See~\cite[Section~5]{AW} for a very similar construction. The right vertical morphism is proper and birational, though some fibers may be positive dimensional. The stacks $\mathfrak M^{\sf exp}_{g,n}(\mathcal A_{B}/B)$ and $\mathfrak M_{g,n}^{\sf ch}(\mathcal A_B/B)$ need not be normal.
\end{remark}

\subsection{Logarithmic geometry}

We finally come to the most well-known approach to the GW theory of degenerations -- logarithmic Gromov--Witten theory. The preceding discussion is meant to give a genuinely geometric picture of what is going on, but it will turn out to be essentially equivalent to a very elegant categorical approach using logarithmic geometry. This theory was first proposed by Siebert, and established in the landmark papers~\cite{AC11,Che10,GS13}. 

We will start with a brief introduction to logarithmic structures, but refer the reader to~\cite{ACGHOSS,ACMUW} for introductions to the subject. 

Here is one useful perspective on logarithmic structures: a scheme
carries polynomial functions, while a logarithmic scheme remembers, in addition,
certain distinguished monomial functions. A key caveat is that the monomials should not be viewed as a subset but as a separate structure with a map to the ring of polynomials. Since monomials mix multiplicatively, the appropriate structure is that of a monoid or a semigroup. Morphisms should include the data of a pullback that takes monomials to monomials, and should commute with the map to the structure sheaf. All the geometric objects we are discussing -- nodal curves, snc degenerations, moduli spaces of curves -- come with natural divisors, e.g. the functions in the total space of a degeneration cutting out a component or the locus in the moduli space of curves comprising singular divisors. The monomials of interest will be the functions cutting out these divisors.

\subsubsection{Logarithmic schemes}

Let $X$ be a scheme. A logarithmic structure on a scheme $X$ consists of a sheaf
of commutative monoids $\mathcal M_X$ with a morphism of sheaves of
monoids
\[
\alpha \colon \mathcal M_X \longrightarrow \mathcal O_X,
\]
where $\mathcal O_X$ is viewed as a monoid under multiplication; it should saisfy the property that the induced map
\[
\alpha^{-1}(\mathcal O_X^\times) \xrightarrow{\;\sim\;} \mathcal O_X^\times
\]
is an isomorphism. A logarithmic scheme is a pair $(X,\mathcal M_X)$ consisting of a
scheme and a log structure.

The examples essentially fall into two categories. The first is the ``geometric'' examples, where the sheaf $\mathcal M_X$ is a subsheaf of the structure sheaf, so monomials are ``polynomials that satisfy a property''. The second are the examples with ``generic logarithmic structure'', where the sheaf $\mathcal M_X$ contains elements that go to $0$ under the map to $\mathcal O_X$. 

\begin{example}[Geometric examples]
Let $X$ be a variety and $D\subset X$ a divisor. The {\it divisorial logarithmic structure} is the sheaf given by
\[
\mathcal M_X(U):= \left\{f\in\mathcal O_X(U)\colon f|_{U\setminus D}\in\mathcal O^\star_X(U\setminus D)\right\}.
\]
In other words, it is those functions on $U$ that become invertible away from $D$. In this example, the sheaf $\mathcal M_X$ is a subsheaf of $\mathcal O_X$. 

The key example among these examples is affine space with the coordinate hyperplanes:
\[
X = \mathbb A^n, \ \ \ D = \bigcup_{i=1}^n \left\{x_i = 0\right\}.
\]
In this case, the monoid $\mathcal M_X(U)$ is precisely the set of monomials
\[
\mathcal M_X(U) = \left\{\alpha\cdot \prod_i x_i^{a_i}\colon \alpha\in\mathcal O^\star_X(U), \ \ a_i\in\NN_0\right\}. 
\]
In particular, $\mathcal M_X(X)$ is exactly the set of monomials in the polynomial ring. We use $\mathbb N_0$ for the natural numbers including $0$. 

One can take $X$ to be any toric variety and $D$ the toric boundary divisor. On a torus invariant affine $U$, the sheaf $\mathcal M_X(U)$ are exactly the characters on the torus that are regular on $U$, up to units. 
\end{example}

We record a non-geometric example. 

\begin{example}[Generically nontrivial logarithmic structure]
Let $D\subset X$ be a smooth divisor in $X$. We can equip $X$ with the divisorial logarithmic structure, and call it $M_X$. We can then look at the morphism
\[
\iota\colon D\hookrightarrow X
\]
and consider the pullback $\iota^\star M_X$. This gives an interesting non-geometric logarithmic structure $M_D$ on $D$. On an open set $U$, the sections of $M_D(U)$ can be identified with $\mathbb N_0\oplus\mathcal O_D^\star(U)$. An element $(n,u)$ maps to $0$ if $n>0$ and to $u$ otherwise. Indeed, on the ambient $X$, the nontrivial elements in the logarithmic structure are powers of the local defining equation for $D$. When restricted to $D$ this function -- by definition -- vanishes. The logarithmic structure ``remembers'' some features of the way in which $D$ is embedded inside $X$. In fact, the structure $M_D$ is equivalent to the data of the normal bundle of $D$ in $X$. 

In particular, we flag that the logarithmic structure in this case is generically nontrivial, meaning it is not isomorphic to $\mathcal O_X^\star$, even generically. This makes it very different from the previous example. 
\end{example}

As we will see shortly, this second kind of example arises naturally by restricting the divisorial logarithmic structure of $(X,D)$ to $D$. This kind of restriction/pullback operation makes the logarithmic category more flexible than, say, toroidal embeddings~\cite{KKMSD}. 

\subsubsection{Logarithmic morphisms}

We turn to morphisms, which of course play a key role in our story. A morphism of log schemes
\[
f \colon (X,\mathcal M_X) \longrightarrow (Y,\mathcal M_Y)
\]
is a morphism of schemes $f \colon X \to Y$ together with a morphism of sheaves of monoids
\[
f^\flat \colon f^{-1}\mathcal M_Y \longrightarrow \mathcal M_X
\]
such that
\[
\begin{tikzcd}
f^{-1}\mathcal M_Y \arrow{r}{f^\flat} \arrow[swap]{d}{f^{-1}\alpha_Y} &
\mathcal M_X \arrow{d}{\alpha_X} \\
f^{-1}\mathcal O_Y \arrow{r} & \mathcal O_X
\end{tikzcd}
\]
commutes, where the bottom arrow is the morphism induced by $f$.

In other words, a morphism of log schemes is a morphism of the underlying
schemes together with a compatible map between their distinguished
monomials.

\begin{example}
Suppose $(X,\mathcal M_X)$ and $(Y,\mathcal M_Y)$ are logarithmic schemes with divisorial logarithmic structures determined by $D$ and $E$. Since the logarithmic sheaves $\mathcal M_X$ or $\mathcal M_Y$ are subsheaves of the structure sheaf, given a scheme theoretic map
\[
f\colon X\to Y
\]
there is at most one logarithmic morphism, so it makes sense to ask \underline{is the map $f$ logarithmic?} This is the case if and only if $f^{-1}(E)\subset D$. 
\end{example}

\subsubsection{The categorical aspects}

Logarithmic structures form a category, but in fact, there are many flavours of these categories. 

Let $(X,\mathcal M_X)$ be a logarithmic scheme. The \underline{characteristic sheaf} is the quotient
\[
\overline{\mathcal M}_X := \mathcal M_X / \mathcal O_X^\times.
\]
The logarithmic structure sheaf $\mathcal M_X$ captures monomials, while the characteristic sheaf captures ``monomials up to scalars''. 

Let $f \colon X \to Y$ be a morphism of schemes, and let
$(Y,\mathcal M_Y)$ be a log scheme. The \underline{pullback log structure}
$f^\star\mathcal M_Y$ on $X$ is defined as the sheafification of the
pushout
\[
f^{-1}\mathcal M_Y \;\oplus_{\,f^{-1}\mathcal O_Y^\times}\; \mathcal O_X^\times,
\]
where $f^{-1}\mathcal O_Y^\times \to \mathcal O_X^\times$ is induced by $f$. Essentially, we are declaring that the pullback of a monomial is a monomial, and then enforcing compatibility with the units. This is the smallest logarithmic structure that makes the arrow logarithmic. A morphism $X\to Y$ of logarithmic schemes is \underline{strict} if the map from the logarithmic structure on $X$ to the pullback of the structure from $Y$ is an isomorphism. 

A standard categorical restriction is \underline{coherence}. 
A logarithmic scheme $(X,\mathcal M_X)$ is \emph{coherent} if for every point $x\in X$, the characteristic monoid $\overline{\mathcal M}_{X,x}$ is finitely generated and there exists an \'etale neighborhood $U \to X$ of $x$ together with a strict morphism
\[
U \longrightarrow \Spec \mathbb{C}[\overline{\mathcal M}_{X,x}],
\]
endowed with its canonical logarithmic structure, inducing an isomorphism on characteristic monoids at $x$.

The log structure $\mathcal M_X$ is
called \underline{fine} if it is coherent and, for every geometric point $\bar x \to X$, the stalk
\[
\overline{\mathcal M}_{X,\bar x}
\]
is a finitely generated integral monoid; integral means that it injects into its groupification.  It is called \underline{fine and saturated} (fs) if it is fine and, in addition,
each stalk is saturated in its groupification, i.e.\ for any
$m \in \overline{\mathcal M}_{X,\bar x}^{\mathrm{gp}}$ with $n m \in
\overline{\mathcal M}_{X,\bar x}$ for some positive integer $n$, we have
$m \in \overline{\mathcal M}_{X,\bar x}$.

These categories also admit fiber products, although we should warn the reader that the fiber product of fine logarithmic schemes may not have the same underlying space as the scheme theoretic fiber product.

All these definitions are straightforward to extend to stacks, working in the \'etale, lisse-\'etale, or flat topologies. 

\subsubsection{Logarithmic smoothness}

\underline{Logarithmic smoothness} is a property that can be held by a morphism $(X,\mathcal M_X)\to (Y,\mathcal M_Y)$. It can be defined using a version of an infinitesimal lifting criterion. We refer the reader to~\cite{ACGHOSS} for further discussion, but here we include a few facts without proof to help the reader. The reader might be reassured to know that the condition (iii) below characterizes the notion. 

\begin{enumerate}[(i)]
\item Let $(X,\mathcal M_X)$ be a divisorial log scheme with divisor $D$.  
If $X$ is a toric variety and $D$ is the full toric boundary, then the map
\[
(X,\mathcal M_X) \longrightarrow \operatorname{Spec} k
\]
is logarithmically smooth. If $X$ is smooth and $D$ is an snc divisor,
the map to $\operatorname{Spec} k$ is also log smooth. More generally, any
toroidal embedding with its toroidal divisor is logarithmically smooth.

\item A dominant equivariant morphism of toric varieties is logarithmically
smooth with respect to the toric log structures. In particular, the blowup
of a stratum of an snc pair, or more generally the blowup of any ideal
generated by monomials in the divisor of an snc pair $(X,D)$, defines a
logarithmically smooth map. Note that logarithmically smooth maps can fail
to be flat and may have non-reduced fibers.

\item \underline{Local structure.} A morphism $(X,\mathcal M_X) \to (Y,\mathcal M_Y)$
is log smooth if, locally on the source and target, it is the pullback of a
dominant equivariant morphism of toric varieties, compatibly with the
toric log structures.

\item Endow the moduli stack of stable curves $\Mbar_g$
with its divisorial log structure coming from the divisor of singular curves,
and endow the universal curve $\mathcal C = \Mbar_{g,1}$
with the induced log structure. Then the morphism
\[
\mathcal C \longrightarrow \Mbar_g
\]
is logarithmically smooth. An analogous statement holds for the stack of
nodal curves. In particular, any family of nodal curves over a base can be
equipped with a logarithmic structure making it logarithmically smooth.

\item \underline{Logarithmic tangent spaces.} If $f \colon (X,\mathcal M_X) \to (Y,\mathcal M_Y)$ is logarithmically
smooth, then it has a relative logarithmic tangent sheaf which is a vector bundle. 
\end{enumerate}

In particular, an snc degeneration
\[
\mathcal X\to B
\]
of the kind we have been studying is logarithmically smooth! Since nodal curves are also logarithmic, this suggests that within the logarithmic category, stable maps from curves to degenerations should behave just like stable maps to smooth targets. 

\subsection{The stack of logarithmic maps} 

We summarize the results of Abramovich--Chen~\cite{AC11,Che10} and Gross--Siebert~\cite{GS13}. We work in the category $\mathsf{LogSch}$ of fine and saturated logarithmic schemes. We suppress the logarithmic structure sheaf $\mathcal M_X$ from logarithmic schemes in what follows, and just spell out when the logarithmic structure is being ignored. 

\begin{definition}
A logarithmically smooth curve is a logarithmically smooth morphism $\pi\colon C\to S$ whose underlying map is flat with reduced and connected fibers of pure dimension $1$. 
\end{definition}

F. Kato gives a complete description of these objects. The underlying map 
$\pi: C \to S$ is a flat family of nodal curves with a collection of sections 
$\{p_i\}$. The logarithmic structure $\mathcal{M}_C$ on $C$ behaves in three ways relative to the base $S$:  

\begin{enumerate}
    \item \underline{Generic points.} At a generic point $\eta \in C$, 
    \[
        \mathcal{M}_{C,\eta} \cong \pi^\star \mathcal{M}_{S,\pi(\eta)}.
    \]
    \item \underline{Nodes.} At a node $q \in C$, the characteristic monoid 
    $\overline{\mathcal{M}}_{C,q} = \mathcal{M}_{C,q}/\mathcal{O}_{C,q}^\times$ 
    is generated by elements $x$ and $y$ corresponding to the two branches, 
    with the relation
    \[
        x y = \pi^\star(t)
    \]
    for some $t \in \mathcal{M}_{S,\pi(q)}$.
    \item \underline{Marked points.} At a marked section $p_i$, the logarithmic 
    structure is generated by a single monomial defining the divisor 
    $p_i(S) \subset C$.
\end{enumerate}

Thus, logarithmic curves are exactly flat families of nodal, marked curves 
equipped with a logarithmic structure encoding nodes and markings as above.

\begin{definition}
Let $\mathcal X\to B$ be a simple normal crossings degeneration, and equip $\mathcal X$ and $B$ with divisorial logarithmic structures coming from $\mathcal X_0$ and $0$ respectively. A logarithmic stable map is a diagram of logarithmic schemes
\[
\begin{tikzcd}
C\arrow{d}{\pi}\arrow{r}{f}& \mathcal X \ar[d]\\
S\arrow{r} & B
\end{tikzcd}
\]
where $\pi$ is a logarithmically smooth curve.
\end{definition}

Fix a curve class $\beta$, the  \underline{moduli stack of logarithmic stable maps} 
\[
\Mbar^{\sf log}_{g,n}(\mathcal{X}/B,\beta)
\] 
is the moduli problem whose fiber over a logarithmic scheme $S$ is the groupoid of diagrams
\[
\begin{tikzcd}
C \arrow{r}{f} \arrow{d}{\pi} & \mathcal{X} \arrow{d}\\
S \arrow{r} & B
\end{tikzcd}
\]
where
\begin{enumerate}[(i)]
    \item $\pi: C \to S$ is a logarithmically smooth curve of genus $g$ with $n$ marked sections,
    \item $f: C \to \mathcal{X}$ is a logarithmic morphism over $B$,
    \item and the underlying map $({C} \to {S}, {f})$ is stable in the usual sense and has class $\beta$.
\end{enumerate}

The key representability theorem of Abramovich--Chen and Gross--Siebert is the following. 

\begin{theorem}
The moduli problem $\Mbar^{\sf log}_{g,n}(\mathcal{X}/B,\beta)$ is representable by a Deligne--Mumford stack equipped with logarithmic structure. 
\end{theorem}

We do not say much about the proof here, but we do note that it is far from formal. While the moduli problem we have defined is naturally defined as a stack over logarithmic schemes, it is not at all obvious that this comes from a stack over the category of schemes, promoted to logarithmic schemes by the choice of a logarithmic structure on that stack. To prove it, the authors develop techniques for characterizing when this kind of representability holds. In fact, the analogous moduli problem for the logarithmic Hilbert scheme \underline{fails} this representability property~\cite{MR20}. 

The moduli stack has a virtual structure. We rephrase this using the results in Abramovich--Wise~\cite{AW}. Let $\mathcal A_B\to B$ be the universal degeneration. The analogous moduli problem for maps from logarithmic curves to $\mathcal A_B$, ignoring the stability condition, is representable by an Artin stack equipped with logarithmic structure, denoted $\mathfrak M_{g,n}^{\sf log}(\mathcal A_B/B)$. 

\begin{theorem}
The stack $\mathfrak M_{g,n}^{\sf log}(\mathcal A_B/B)\to B$ is flat with irreducible, normal total space. The morphism
\[
\Mbar^{\sf log}_{g,n}(\mathcal{X}/B,\beta)\to \mathfrak M_{g,n}^{\sf log}(\mathcal A_B/B)
\]
is equipped with a relative perfect obstruction theory over $B$. Since the fibers of $\mathfrak M_{g,n}^{\sf log}(\mathcal A_B/B)\to B$ carry fundamental classes, we obtain, for each $b\in B$, a virtual class $[\Mbar^{\sf log}_{g,n}(\mathcal{X}_b,\beta)]^{\sf vir}$. 
\end{theorem}

\subsection{Comparing the approaches}

We now have three moduli spaces of logarithmic maps -- the cheap version $\Mbar^{\sf ch}_{g,n}(\mathcal{X}/B,\beta)$, the expanded version $\Mbar^{\sf exp}_{g,n}(\mathcal{X}/B,\beta)$, and the logarithmic version $\Mbar^{\sf log}_{g,n}(\mathcal{X}/B,\beta)$. The map
\[
\mathfrak M_{g,n}^{\sf log}(\mathcal A_B/B)\to \mathfrak M_{g,n}^{\sf ch}(\mathcal A_B/B)
\]
is known to be quasi-finite and proper by~\cite{ACMW} and the domain is normal as we said above. It is therefore automatically the normalization. The geometric spaces
\[
\Mbar^{\sf log}_{g,n}(\mathcal{X}/B,\beta)\to\Mbar^{\sf ch}_{g,n}(\mathcal{X}/B,\beta)
\]
are pulled back from the preceding maps under $\Mbar^{\sf ch}_{g,n}(\mathcal{X}/B,\beta)\to \mathfrak M_{g,n}^{\sf ch}(\mathcal A_B/B)$. Note that the pullback of the normalization is typically not the normalization. 

The map 
\[
\Mbar^{\sf exp}_{g,n}(\mathcal{X}/B,\beta)\to\Mbar^{\sf ch}_{g,n}(\mathcal{X}/B,\beta)
\]
as we noted is not finite; it can have positive dimensional fibers and is not normal. This is again the pullback of an appropriate map on stacks of maps to $\mathcal A_B/B$. Also as we noted earlier, the stack $\mathfrak M_{g,n}^{\sf exp}(\mathcal A_B/B)$, even in the Jun Li setting, is not normal. It does have a normalization with a modular interpretation via logarithmic structures~\cite{R19}. This leads to a fourth space $\Mbar^{\sf log-exp}_{g,n}(\mathcal{X}/B,\beta)$. Although we don't say too much more about it, it can be constructed by studying logarithmic maps to the universal expansion whose underlying map is predeformable. Summarizing:

\[
\begin{tikzcd}[row sep=1.5em, column sep=1.5em, every label/.append style={font=\small}, scale=0.85]
& \Mbar^{\sf log-exp}_{g,n}(\mathcal{X}/B,\beta)
  \arrow[dl] \arrow[dr] & \\
\Mbar^{\sf log}_{g,n}(\mathcal{X}/B,\beta)
  \arrow[dr] &&
\Mbar^{\sf exp}_{g,n}(\mathcal{X}/B,\beta)
  \arrow[dl] \\
& \Mbar^{\sf ch}_{g,n}(\mathcal{X}/B,\beta) &
\end{tikzcd}
\qquad
\begin{tikzcd}[row sep=1.5em, column sep=1.5em, every label/.append style={font=\small}, scale=0.85]
& \mathfrak M_{g,n}^{\sf log-exp}(\mathcal A_B/B)
  \arrow[dl] \arrow[dr] & \\
\mathfrak M_{g,n}^{\sf log}(\mathcal A_B/B)
  \arrow[dr] &&
\mathfrak M_{g,n}^{\sf exp}(\mathcal A_B/B)
  \arrow[dl] \\
& \mathfrak M_{g,n}^{\sf ch}(\mathcal A_B/B) &
\end{tikzcd}
\]
In the right diamond, all four spaces are birational. The two spaces with log as part of the superscript are normal, and the lower left arrow is normalization. The top right arrow is also normalization. The diamond on the left is obtained by pulling back along $\Mbar^{\sf ch}_{g,n}(\mathcal X/B,\beta)\to \mathfrak M_{g,n}^{\sf ch}(\mathcal A_B/B)$. 

All the spaces except the cheap one have natural moduli theoretic interpretations. In fact, the cheap space also has a moduli interpretation -- Wise has studied logarithmic mapping spaces in a very general context~\cite{Wis16b,Wis16a}, and it follows from these papers that our cheap logarithmic space coincides with the \underline{fine} (as opposed to fine and saturated) logarithmic mapping space. 

\begin{remark}[Abramovich--Fantechi]
In fact, there is a fifth space, although it is isomorphic to the fourth space, following Abramovich--Fantechi~\cite{AF11}. Specifically, one can add orbifold structure to the nodes and markings of the curve, and the singularities of the target. The details of this have not been checked, but from the smooth pair case, one would expect that it coincides with the logarithmic-expanded space~\cite{R19}. 
\end{remark}

\subsection{Tropical stratifications}

Under mild hypotheses, any logarithmic scheme or stack carries a natural stratification—morally, the loci where the logarithmic structure (more precisely, its characteristic sheaf) jumps. Under additional assumptions, logarithmic morphisms induce maps of strata. We describe this stratification for the space of logarithmic maps to a degeneration. 

As in the study of the moduli space of curves or stable maps, a practical understanding of logarithmic stable maps relies on controlling this stratification. In our setting, it is encoded by \underline{combinatorial types of tropical maps}. Analogous stratifications exist for the four variants described above; we focus on 
\[
\mathfrak M_{g,n}^{\sf log}(\mathcal A_B/B),
\]
and then explain the necessary modifications in the other cases.

\subsubsection{Stratification in theory}

Before giving the formal definition, note that the log structure $\mathcal A_B/B$ induces a map of fans
\[
\RR^r_{\geq 0} \longrightarrow \RR_{\geq 0},
\]
given by summing coordinates. This is the toric picture for the standard smoothing of the $r$-fold intersection point $ \mathbb V(\prod_{i=1}^r x_i-t)\to \A^1_t$. 

\begin{definition}[Tropical curve]
A \underline{tropical curve} is the dual graph of a prestable curve, together with a choice of positive real edge lengths for each finite edge. Vertices may carry genus labels and marked points are treated as half-edges, as usual. 
\end{definition}

\begin{definition}[Tropical map]
A \underline{tropical map} to $\RR^r_{\geq 0}$ consists of a tropical curve together with a continuous piecewise-linear map of polyhedral complexes to $\RR^r_{\geq 0}$ that is \underline{vertical}, meaning that it is contracted under
\[
\RR^r_{\geq 0} \longrightarrow \RR_{\geq 0}.
\]
The map is constant on all half-edges corresponding to marked points\footnote{Note this final condition will change in the pairs case that we discuss later -- half-edges with non-constant slope are hiding as halves of edges with non-constant slope.}. 
\end{definition}

Equivalently, one can view such a map as a piecewise-linear map to a chosen dilation $\ell\cdot\Delta_r$ of the $r$-simplex, see Figure~\ref{fig: tropical-map}. The number $\ell$ is the image in $\RR_{\geq 0}$ to which the curve is contracted under the projection.
\begin{figure}[h!]
\begin{center}
\includegraphics[scale=0.5]{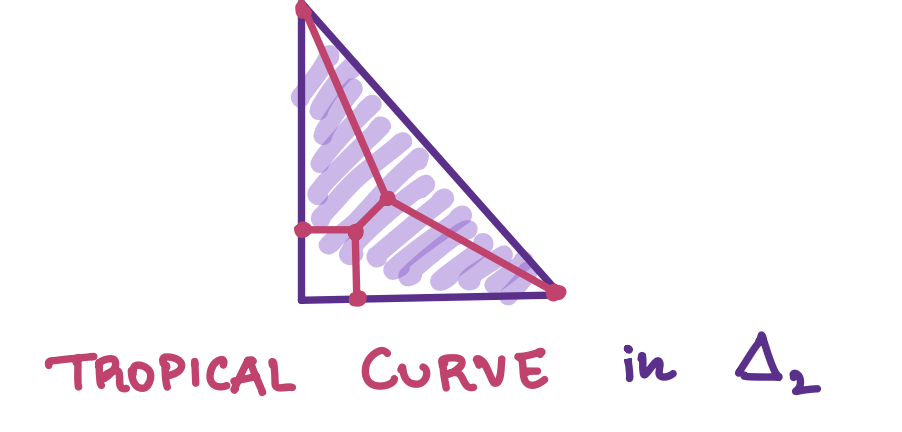}
\end{center}
\caption{A tropical map to a dilated simplex.}\label{fig: tropical-map}
\end{figure}

The set of all such tropical maps can be parameterized by a \underline{cone stack}—that is, a stack over the category of cone complexes. This theory is carefully developed in~\cite{CCUW}. Roughly speaking, it is obtained by gluing cones together, although the gluing diagram is slightly more general than what is allowed for ordinary cone complexes, permitting certain self-gluings and automorphisms.

Once this foundational work is in place, the parameterizing object for tropical maps, which we denote
\[
\mathsf T_{g,n}(\RR^r_{\geq 0}/\RR_{\geq 0}),
\]
can be constructed in a fairly naive manner. We give a brief and informal summary. A tropical map has \underline{discrete parameters} and \underline{continuous parameters}. The discrete parameters consist of the underlying dual graph, the cones of $\mathbb R_{\geq 0}^r$ to which each vertex maps, and the directions of the edges. The continuous parameters are the vertex positions with the cones in which they reside and the edge lengths.

Fixing the discrete parameters, piecewise linearity and continuity impose nontrivial conditions on the vertex and edge parameters. These conditions cut out a cell inside a cone of allowable choices. The discrete parameters carry a natural specialization order (for instance, edges of graphs may be contracted, and cones may be replaced by their faces). These specializations induce attaching maps between cells, and with some care one obtains the cone stack $\mathsf T_{g,n}(\RR^r_{\geq 0}/\RR_{\geq 0})$. We again refer to the beautiful article~\cite{CCUW} for a more complete treatment of the foundations and~\cite{R19} specifically for tropical maps.

Via the Artin fan construction, one obtains a $0$-dimensional Artin stack
\[
a^\star\mathsf T_{g,n}(\RR^r_{\geq 0}/\RR_{\geq 0})
\]
equipped with a map to $[\A^1/\mathbb G_m]$. A key fact in the subject is that there are natural maps
\[
\Mbar^{\sf log}_{g,n}(\mathcal{X}/B,\beta)
\longrightarrow
\mathfrak M_{g,n}^{\sf log}(\mathcal A_B/B)
\longrightarrow
a^\star\mathsf T_{g,n}(\RR^r_{\geq 0}/\RR_{\geq 0}).
\]
The last space is discrete—it consists of one point for each combinatorial type—and hence is naturally stratified. Indeed, its underlying topological space is essentially the geometric realization of the poset (more precisely, the category of faces) of $\mathsf T_{g,n}(\RR^r_{\geq 0}/\RR_{\geq 0})$. Pulling back along this map yields a stratification of the space $\Mbar^{\sf log}_{g,n}(\mathcal{X}/B,\beta)$ of logarithmic stable maps. Note that according to this indexing many (in particular, all but finitely many) strata will be empty. There is not claim of surjectivity onto the final space in the maps above. 

\medskip

\noindent
\underline{What does this map encode?}
The morphism
\[
\Mbar^{\sf log}_{g,n}(\mathcal{X}/B,\beta)
\longrightarrow
a^\star\mathsf T_{g,n}(\RR^r_{\geq 0}/\RR_{\geq 0})
\]
is a map of algebraic stacks whose construction we do not explain here. However, since the target is discrete, we can describe informally the information it records.

Some aspects are immediate. Given a logarithmic stable map, we can record the dual graph of the domain and, for each vertex, the cone corresponding to the stratum of $\mathcal X$ to which the component maps. This already determines part of the associated combinatorial type.

A subtler feature is the slope and direction of an edge of the tropical curve. More precisely, for a tropical map each edge has a well-defined ``direction'': it maps with some non-negative integer slope along a line in $\RR_{\geq 0}^r$. To understand what this encodes, first suppose we are given a map
\[
C \to \mathcal X_0
\]
that is {predeformable}. An edge then corresponds to a node $q$ of $C$, and by definition of predeformability this node maps either to the interior of a component or to the double locus.

We can describe the associated combinatorial type (i.e.\ the corresponding point of $a^\star\mathsf T_{g,n}(\RR^r_{\geq 0}/\RR_{\geq 0})$) as follows:

\begin{itemize}
\item[(i)] If the node $q$ maps to the interior of a component, then the tropical map is constant along the corresponding edge.

\item[(ii)] If $q$ maps to a component of the double locus, recall that such components are in bijection with the edges of the simplex. In this case, the edge corresponding to $q$ must map to the edge of the simplex corresponding to that component of the double locus.

\item[(iii)] In the situation of (ii), the slope with which the tropical edge maps is precisely the (well-defined) tangency order of $q$ with the double locus.
\end{itemize}

In the general case, when we have a logarithmic map $C \to \mathcal X_0$ that is not predeformable, one can still extract the edge direction in two ways. One approach uses the logarithmic structure directly, as explained beautifully in~\cite{GS13}. Geometrically, one may instead factor the map $C \to \mathcal X_0$ through an expansion $\mathcal X_0'$ until it becomes predeformable. As we have seen, such an expansion is determined by a polyhedral subdivision of the simplex, and one obtains a similar picture involving the edges of this subdivision.

\begin{remark}
In the case of the moduli space of curves $\Mbar_{g,n}$, the space $\mathsf T_{g,n}$ is the cone complex of stable tropical curves, and the topological space underlying $a^\star\mathsf T_{g,n}$ is the set of stable graphs under specialization by edge contraction. The resulting stratification is the usual one induced by the normal crossings boundary divisor. See~\cite{CCUW} for an honest and insightful treatment. 
\end{remark}

\subsubsection{Stratification in practice}

As defined above, the cone complex $\mathsf T_{g,n}(\RR^r_{\geq 0}/\RR_{\geq 0})$ is very large—it has infinitely many faces of arbitrarily large dimension. Of course, most of these are not geometrically relevant. One can replace $\mathsf T_{g,n}(\RR^r_{\geq 0}/\RR_{\geq 0})$ with a finite type object that retains (or refines) the stratification. We will not go into the details here, but instead summarize the details here and direct the reader to the discussion of boundedness of combinatorial types in~\cite{NS06} and~\cite{GS13}. The finiteness comes from the following three constraints:

\begin{enumerate}[(i)]
\item First, since $\beta$ is fixed, one may decorate each vertex with an effective curve class in such a way that the sum over all vertices is $\beta$.

\item Next, one imposes a stability condition: any vertex whose curve class decoration is $0$ must be stable.

\item Finally, the slopes of the piecewise-linear map to $\RR_{\geq 0}^r$ along edges incident to a given vertex are geometrically constrained by the curve class at that vertex. This is called the \underline{balancing condition}, and should be familiar to tropical geometers.
\end{enumerate}

A key result, proved in~\cite{GS13,NS06}, is that after imposing these conditions one obtains a finite type subcomplex of $\mathsf T_{g,n}(\RR^r_{\geq 0}/\RR_{\geq 0})$ consisting of the geometrically meaningful faces. Although the resulting complex remains extremely complicated, it can be handled in practice using the techniques of tropical geometry.

\subsubsection{Stratification in the expanded theory}

There are analogous stratifications in the expanded theory 
$\Mbar^{\sf exp}_{g,n}(\mathcal{X}/B,\beta)$. 
The most natural one—almost tautological—arises from the stack of expansions 
$\mathsf{Exp}_{\mathcal X/B}$ itself. Indeed, this stack is constructed from a cone stack to begin with, and there is a tautological morphism
\[
\Mbar^{\sf exp}_{g,n}(\mathcal{X}/B,\beta)
\longrightarrow 
\mathsf{Exp}_{\mathcal X/B}.
\]
This endows $\Mbar^{\sf exp}_{g,n}(\mathcal{X}/B,\beta)$ with a stratification indexed by cones in a space of $1$-complexes. 

In practice, it is more convenient to ``blend'' the stratification coming from the stack of expansions with the tropical one. This is straightforward to do -- given a $1$-complex $\mathcal P$ in $\ell\cdot\Delta_r$, we consider the additional moduli of a piecewise linear parameterization
\[
\Gamma\to \mathcal P
\]
whose vertices map exactly onto the vertices of $\mathcal P$ and whose edges are either contracted or map to the edges, i.e. the polyhedral complex $\mathcal P$ is the ``image'' of $\Gamma$. 

After a choice of stack of expansion, there is a corresponding cone stack $\mathsf T^{\sf exp}_{g,n}(\RR_{\geq 0}^r/\mathbb R_{\geq 0})$. 

The cones in this cone complex index the strata of the space $\Mbar^{\sf exp}_{g,n}(\mathcal X/B,\beta)$. As before, though, according to this indexing many (in particular, all but finitely many) strata will be empty. The procedures here are largely insensitive to whether a particular stratum is empty or not. 

As before, geometric considerations ensure that only finitely many points in the target are actually relevant. We refer to~\cite{MR23} for further details.

\section{Curves in pairs}

Let us begin with an snc degeneration $\mathcal X\to B$. We have constructed a space 
$\Mbar_{g,n}^{\sf log}(\mathcal X_0,\beta)$ 
of stable maps to the special fiber. Our goal is to understand this space, and its virtual class, in terms of the components and strata of $\mathcal X_0$.

\subsection{What should the degeneration formula say?}

The geometric idea is straightforward. Suppose first that we are given a predeformable map
\[
C_0\to \mathcal X_0.
\]
After normalizing the domain, we obtain maps (possibly with disconnected domains)
\[
C_i\to X_i,
\]
where the $X_i$ are the components of $\mathcal X_0$. One would like to reverse this procedure and show that any map 
$C_0\to \mathcal X_0$ 
can be reconstructed by gluing together stable maps to the various $X_i$. A cartoon of the idealized picture is shown in Figure~\ref{fig: basic-deg-picture}.
\begin{figure}[h!]
\begin{center}
\includegraphics[scale=0.4]{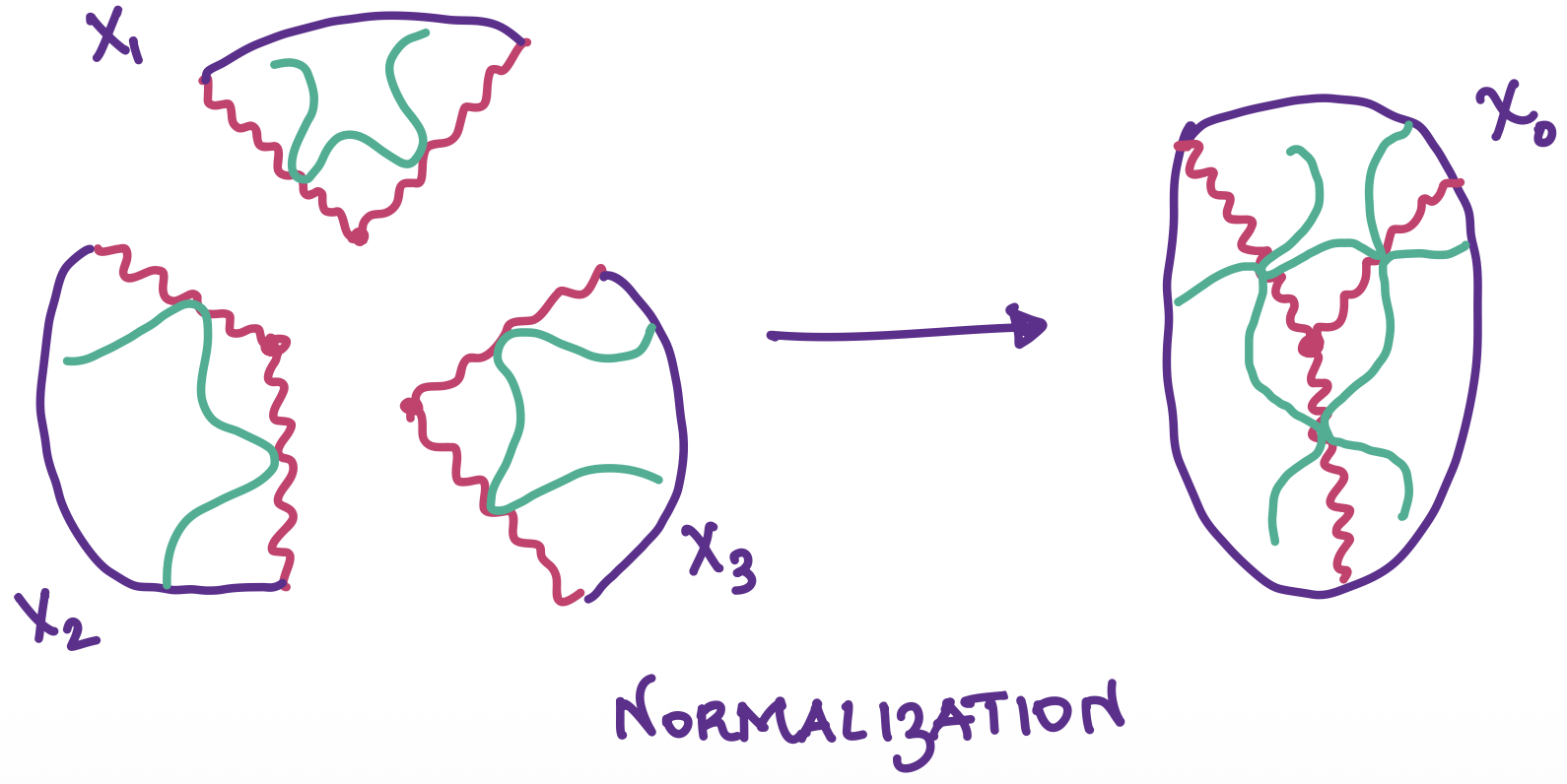}
\end{center}
\caption{A predeformable stable map to the singular fiber and the normalization.}\label{fig: basic-deg-picture}
\end{figure}

There are, however, a few subtleties.

First, given a predeformable point $[C_0\to \mathcal X_0]$ and a node $q$ of $C_0$, the tangency order at $q$ is fixed and in fact locally constant in 
$\Mbar_{g,n}^{\sf log}(\mathcal X_0,\beta)$. 
Thus, in order to glue maps, one must fix the tangency of each $C_i$ along the boundary divisor 
$\partial X_i\subset X_i$, 
where $X_i$ meets the other components of $\mathcal X_0$. 
One must also remember the marked points of $C_i$ at which the gluing will eventually occur. 

Accordingly, we consider maps from nodal curves
\[
f\colon (C,p_1,\ldots, p_m)\to (X|\partial X), 
\qquad 
\partial X = \sum D_j,
\]
such that $f^{-1}(\partial X)\subset \{p_i\}$ and the tangency orders $c_{ij}$ of $f$ at $p_i$ along $D_j$ are fixed. This defines a locally closed—though typically neither open nor closed—substack of the stack of ordinary stable maps. This is the source of the ``fixed tangency'' part of the problem. 

Second, away from the predeformable locus one must decide how to proceed. Once formulated correctly, this issue exactly parallels the situation earlier in the study of maps to $\mathcal X_0$, and again requires enlarging the framework. And in fact, again, there are three equivalent ways out -- cheap, expanded, and logarithmic. 

Ultimately, we are looking for a formula at the level of both spaces and virtual classes that is modelled on the geometric picture we have just described:
\[
\begin{tikzpicture}
\node at (0,1.6) {\textbf{The shape of the degeneration formula}};
\node[
  draw,
  rounded corners=8pt,
  inner sep=12pt
] at (0,0)
{$
\left\{
\begin{array}{c}
\text{Curves in } \mathcal X_0 \\
\text{with fixed tangency}\\
\text{with the double locus}
\end{array}
\right\}
\;=\;
\left(
\prod_i
\left\{
\begin{array}{c}
\text{Curves in } X_i \\
\text{with fixed tangency with $\partial X_i$}
\end{array}
\right\}
\right)
\;\cap\;
\{\text{gluing conditions}\}.
$};
\end{tikzpicture}
\]
The gluing theorem we will prove will exactly be of this form, although we stress that the naive forms of it, i.e. the simplest generalizations of Jun Li's formula~\cite{Li02}, are false!

Note also that the $X_i$ should be understood as irreducible components not just of $\mathcal X_0$ but possibly also in further expansions of it. In other words, $X_i$ is typically a toric compactification of a normal torus bundle over a stratum.

\subsection{Everything from before, done again}

Let $(X|\partial X)$ be a pair consisting of a smooth projective variety and $\partial X$ a simple normal crossings divisor with components $D_1,\ldots, D_r$. We now speedily construct the three versions of the space of stable maps to a pair. 

The reader may wish to keep the following in mind as we go through the discussion. The locus we already know is relevant is the space of maps from pointed curves to $X$, with fixed tangency along the components of $\partial Y$. As the data move in moduli the entire curve, or a component of it, can fall into $\partial X$. Once this happens, there is no obvious way of measuring the contact order. In our earlier discussion, we wanted to understand how to \underline{degenerate} the space $\Mbar_{g,n}(\mathcal X_\eta,\beta)$. Now we want to understand how to \underline{compactify} the space of maps to $X$ with fixed tangency with $\partial X$. 

\subsubsection{The non-degenerate locus}

We fix a genus $g$, a number $n$ of marked points, and a curve class $\beta$. Fix a matrix $[c_{ij}]$ of integers, where $i$ ranges over $\{1,\ldots, n\}$ and $j$ ranges over $\{1,\ldots, m\}$. We use the symbol $\Lambda$ to package all the discrete data $(g,n,\beta,[c_{ij}])$. 

The basic object, which will form the interior of our moduli space, is the space $\mathsf M_\Lambda(X|\partial X)$ whose points are flat families of maps from marked smooth curves
\[
\begin{tikzcd}
(C,p_1,\ldots,p_n)\arrow{d}\arrow{r}{f} & (X|\partial X)\\
S
\end{tikzcd}
\]
where $f^{-1}(\partial X)$ is contained in the union of the marked sections, with tangency, curve class, and genus specified by $\Lambda$. This is a Deligne--Mumford stack, typically non-proper. The virtual/expected dimension of this moduli problem is
\[
\mathsf{vdim} \ \mathsf M_\Lambda(X|\partial X) = (\dim X-3)(1-g)+\int_\beta c_1(T_{X|\partial X}^{\sf log})+n. 
\]
The space does carry a virtual fundamental class in this virtual dimension, but it is not proper. 

\subsubsection{The cheap trick}

For the cheap space, we again have two basic requirements: properness and independence from the global geometry of $X$. 

Properness requires that the stable map limit in $\Mbar_{g,n}(X,\beta)$ of any family in $\mathsf M_\Lambda(X|\partial X)$ should lie in the proposed space. Independence from global considerations can be formulated as follows. Consider a strict morphism $(X|\partial X)\to (Y|\partial Y)$, meaning that every divisor component in $\partial X$ is the pullback of a component in $\partial Y$. Then, for any ordinary stable map to $X$, if the composition with $X\to Y$ lies in the cheap logarithmic space, the original stable map should lie there as well.

We now define the moduli space. We label the components of $\partial X$ by the set $\{1,\ldots,r\}$. The stack $\mathcal A^r = [\mathbb A^r/\mathbb G_m^r]$ is the universal target for maps of the form above. A map $f\colon C\to[\mathbb A^r/\mathbb G_m^r]$ is exactly an $r$ tuple of pairs $(L_i,s_i)$ of line bundles and sections. We have fixed $\Lambda$ so in particular a contact order matrix.

 Inside the mapping stack $\mathfrak M_{g,n}(\mathcal A^r)$, we can consider the locus
 \[
 \mathfrak M^\circ_\Lambda(\mathcal A^r|\partial \mathcal A^r)\subset \mathfrak M_{g,n}(\mathcal A^r)
 \]
 where $s_j$ is generically nonzero in a neighborhood of $p_i$, vanishes at $p_i$ to order $c_{ij}$, and the locus where any of the $s_j$ vanish is contained in the set $\{p_i\}$. This is just spelling out what the condition we had for maps to $(X|\partial X)$ for the target $(\mathcal A^r|\partial \mathcal A^r)$.
 
Note that $\mathfrak M^\circ_\Lambda(\mathcal A^r|\partial \mathcal A^r)$ is naturally identified with the stack of smooth pointed curves, so it is irreducible of dimension $3g-3+n$.

\begin{definition}
Let $ \mathfrak M^{\sf ch}_\Lambda(\mathcal A^r|\partial \mathcal A^r)$ be the closure of $ \mathfrak M^\circ_\Lambda(\mathcal A^r|\partial \mathcal A^r)\subset \mathfrak M_{g,n}(\mathcal A^r)$. The stack $\Mbar^{\sf ch}_\Lambda(X|\partial X)$ of cheap logarithmic maps to $(X|\partial X)$ is defined by the pullback diagram:
\[
\begin{tikzcd}
\Mbar^{\sf ch}_{\Lambda}(X|\partial X)\arrow{r}\arrow{d} & \mathfrak M^{\sf ch}_\Lambda(\mathcal A^r|\partial \mathcal A^r) \arrow{d} \\
\Mbar_{g,n}(X,\beta)\arrow{r} & \mathfrak M_{g,n}(\mathcal A^r). 
\end{tikzcd}
\]
\end{definition}

The bottom horizontal arrow is equipped with a relative perfect obstruction theory, see for instance~\cite{ACW,BNR22}. The top right space is irreducible and so we obtain a class
\[
[\Mbar^{\sf ch}_{\Lambda}(X|\partial X)]^{\sf vir}\in \mathsf{CH}_\star(\Mbar^{\sf ch}_{\Lambda}(X|\partial X);\QQ). 
\] 

Again, this is really the smallest theory that we could possibly take that doesn't depend on the global geometry of $(X|\partial X)$ in a serious way. The expanded theory will be the pullback of a birational modification and the logarithmic theory will be the pullback of a normalization. 

\subsubsection{Expansions}

We next describe the approach using expansions. An \underline{expansion} of the pair $(X | \partial X)$ is obtained as the special fiber of a degeneration constructed from the trivial family
\[
X \times \mathbb{A}^1 \to \mathbb{A}^1,
\]
by (after possible base change) iteratively blowing up strata supported in the special fiber.

When $\partial X$ is smooth and we're in the situation studied by Jun Li, this just inserts projective line bundles over $\partial X$, possibly repeatedly. In the snc case the expansions can be more complicated, but each irreducible component of the expansion is an equivariant compactification of a torus torsor over a stratum of $\partial X$. Again, the object is constructed by a combinatorial construction:
\[
\left\{
\begin{array}{c}
\text{polyhedral decompositions of } \RR_{\ge 0}^r \\
\text{with integral vertices}
\end{array}
\right\}
\;\;\longleftrightarrow\;\;
\left\{
\text{expansions of } (\mathcal A^r | \partial \mathcal A^r)
\right\}.
\]
It might be visualized as in Figure~\ref{fig: exps2}.
\begin{figure}[h!]
\begin{center}
\includegraphics[scale=0.4]{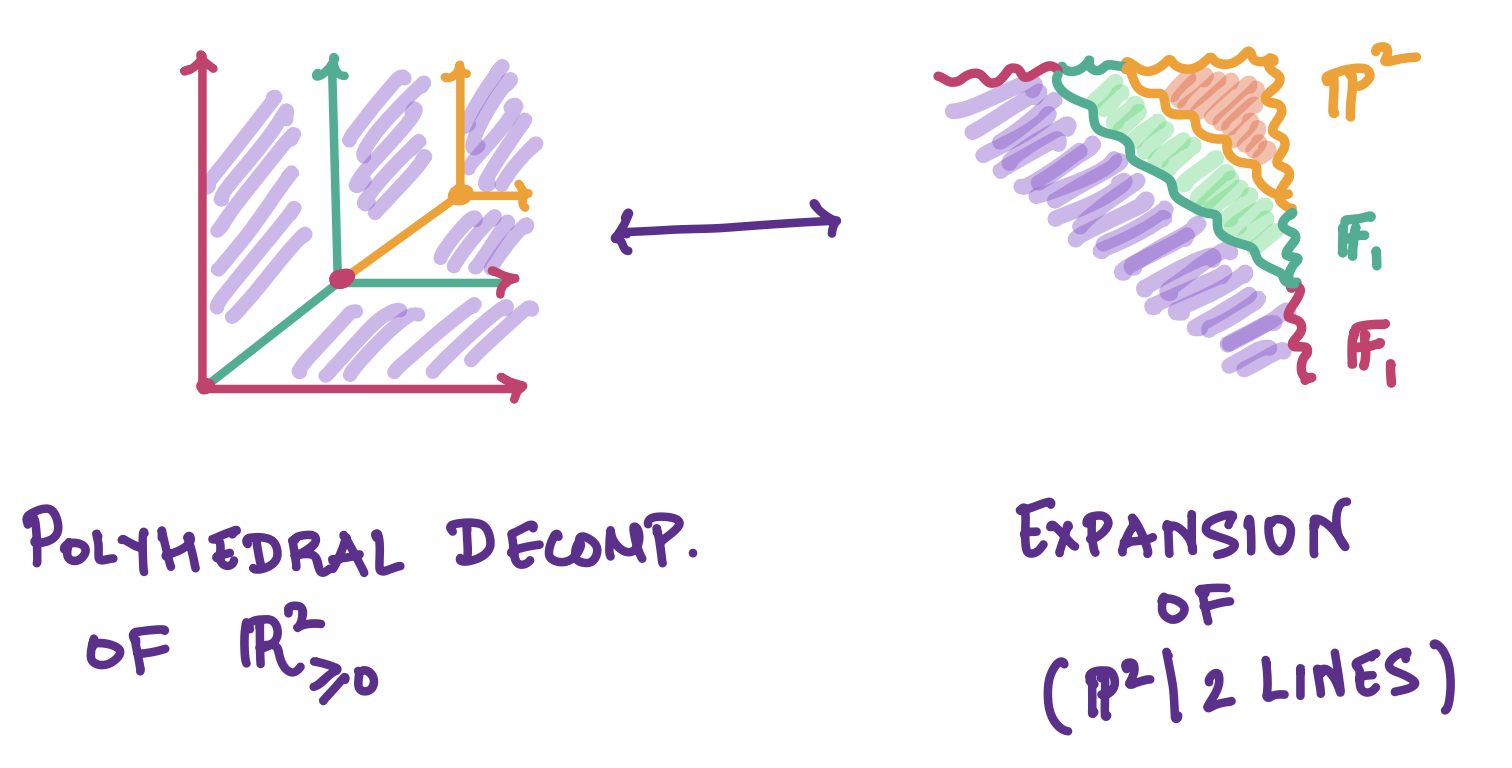}
\end{center}
\caption{An expansion of the pair $\mathbb P^2$ where $\partial \mathbb P^2$ is a union of two lines. The expansion is obtained geometrically by degeneration to the normal cone of the $0$-stratum, and then the same again. The polyhedral decomposition is shown on the left.}\label{fig: exps2}
\end{figure}

Since we have a map
\[
X\to\mathcal A^r
\]
we can pullback to get expansions of $(X|\partial X)$\footnote{Note that some of the intersections of divisors in $\partial X$ can be empty, so expansions of $X$ and $\mathcal A^r$ are not quite in bijection.}. 

As in the case of degenerations, these objects assemble into a stack. 
After fixing a polyhedral decomposition of the relevant tropical parameter space, one obtains a stack
\[
\mathrm{Exp}(X | \partial X)
\]
parametrizing expansions of the pair. Again, there are choices here; the choices form an inverse system and do not affect the invariants or cycles that we will eventually build. There is a family fitting into a diagram
\[
\begin{tikzcd}
\mathcal{X} \arrow[d] \arrow[r] & X \\
\mathrm{Exp}(X | \partial X)
\end{tikzcd}
\]
where the vertical arrow is the universal expanded target and the horizontal arrow is the natural contraction to $X$. The fiber over a point of $\mathrm{Exp}(X | \partial X)$ is the corresponding expansion of $(X|\partial X)$. Some of the components of the fibers are $\mathbb P^1$-bundles and are designated \underline{tube components}. As before, given an expansion $\mathcal X$, we can contract the tube components to get $\mathcal X\to\overline{\mathcal X}$. 

For each component $D_i\subset\partial X$, we can specialize it to any expansion to get a corresponding expanded divisor in any expansion of $X$, denoted $\mathcal D_i$. We will measure tangency along these expanded divisors. 

We can define the moduli space of stable maps to expansions in parallel with the degeneration case. 

\begin{definition}[Expanded maps to pairs]
The stack $\Mbar^{\sf exp}_\Lambda(X|\partial X)$ is the fibered category over schemes whose fiber over $S$ consists of a family of expansions $\mathcal X/S$ together with a map from a nodal curve
\[
f\colon (C,p_1,\ldots,p_n)\longrightarrow \mathcal X \longrightarrow X
\]
to an expansion of $X$ along $\partial X$, subject to the following fiberwise conditions:
\begin{enumerate}[(i)]
    \item The preimage of $\partial \mathcal X$ in $C$ is contained in the set of marked points $\{p_i\}$;
    \item the preimage of the singular locus of $\mathcal X$ is contained in the singular locus of $C$;
    \item the map is predeformable, i.e.\ the tangency order of a node of $C$ with the singular locus of $\mathcal X$, measured after normalization, is equal on the two branches of the node;
    \item for each geometric point $s \in S$, let $T \subset \mathcal X_s$ be a tube component with projection
\[
\pi_T : T \to D
\]
exhibiting it as a $\mathbb P^1$-bundle. If $C'_s \subset C_s$ is an irreducible component such that 
$f(C'_s) \subset T$, then
\[
f|_{C'_s} : C'_s \longrightarrow T
\]
factors through a fiber of $\pi_T$ and is a finite cover of that 
$\mathbb{P}^1$-fiber, ramified only over the sections 
$0$ and $\infty$. In particular, there is a factorization $\overline C_s \to \overline{\mathcal X}_s$ obtained by contracting all tube components and the components of their preimages;

    \item the map $\overline C_s \to \overline{\mathcal X}_s$ has finite automorphism group over $X$.
\end{enumerate}
The tangency of the map along the specialized divisors $\mathcal D_j$ at the point $p_i$ is $c_{ij}$ for all $i$ and $j$, and the curve class is measured after pushforward to $X$.
\end{definition}

We again have the analogous theorem, proved in~\cite{MR23} following~\cite{R19}.

\begin{theorem}
The fibered category $\Mbar^{\sf exp}_\Lambda(X|\partial X)$ is representable by a proper Deligne--Mumford stack. The morphism
\[
\Mbar^{\sf exp}_\Lambda(X|\partial X)\to \mathsf{Exp}(X|\partial X)
\]
is equipped with a relative perfect obstruction theory. The target is irreducible, so there is a virtual class
\[
[\Mbar^{\sf exp}_\Lambda(X|\partial X)]^{\sf vir}\in \mathsf{CH}_\star(\Mbar^{\sf exp}_\Lambda(X|\partial X);\QQ).
\]
\end{theorem}

By recognizing this as the pullback of a proper birational morphism to the stack of cheap prestable maps to $(\mathcal A^r|\partial\mathcal A^r)$, formal properties of virtual intersection theory imply that pushforward along
\[
\Mbar^{\sf exp}_\Lambda(X|\partial X)\to \Mbar^{\sf ch}_\Lambda(X|\partial X)
\]
identifies virtual classes. 

\subsubsection{Logarithmic maps}

The logarithmic approach, being quite categorical in nature, requires very little change. We start with a few words about the discrete data. 

Equip $(X|\partial X)$ with the divisorial logarithmic structure. Let $C/S$ be a logarithmically smooth curve with marked sections $p_i$ and a morphism
\[
\begin{tikzcd}
C\ar[d]\ar[r]& X\\
S.
\end{tikzcd}
\]
Fix a marked point $p_i$ and a divisor $D_j$. The logarithmic structure induces a map of characteristic monoids
\[
\NN\;\cong\; \overline{\mathcal{M}}_{X,D_j,f(p_i)} \;\longrightarrow\; \overline{\mathcal{M}}_{C,p_i} \cong\; \mathbb{N},
\]
and the image of the generator of $\overline{\mathcal{M}}_{X,D_j,f(p_i)}$ is the \underline{contact order} of $f$ along $D_j$ at $p_i$. When the logarithmic structure of $S$ is trivial and we have a map
\[
(C,p_1,\ldots,p_n)\to (X|\partial X)
\]
the contact orders are just the scheme theoretic tangency order. 

It is a basic feature of logarithmic geometry that these contact orders are locally constant in flat families. The genus $g$ and marking $n$ are fixed and the curve class is measured with respect to the underlying map. 

Again, we have a moduli problem $\Mbar^{\sf log}_\Lambda(X|\partial X)$ of logarithmic stable maps. It is a stack on the category of fine and saturated logarithmic schemes. The representability theorem of Abramovich--Chen~\cite{AC11,Che10} and Gross--Siebert~\cite{GS13} is the following. 

\begin{theorem}
The moduli problem $\Mbar^{\sf log}_\Lambda(X|\partial X)$ is represented by a proper Deligne--Mumford stack equipped with logarithmic structure. 
\end{theorem}

The authors prove there is a virtual structure. One way to describe it, slightly different from how it was first presented, is as follows. There is again a ``universal target' version of the moduli space, denoted $\mathfrak M^{\sf log}_\Lambda(\mathcal A^r|\partial \mathcal A^r)$. Paraphrasing Abramovich--Wise from~\cite{AW}, we have:

\begin{theorem}
The morphism 
\[
\Mbar^{\sf log}_\Lambda(X|\partial X)\to \mathfrak M^{\sf log}_\Lambda(\mathcal A^r|\partial \mathcal A^r)
\]
is equipped with a relative perfect obstruction theory leading to a virtual class
\[
[\Mbar^{\sf log}_\Lambda(X|\partial X)]^{\sf vir}\in \mathsf{CH}_\star(\Mbar^{\sf log}_\Lambda(X|\partial X);\mathbb Q). 
\]
\end{theorem}

\subsubsection{Tropical stratification}

Again, there is a tropical stratification, with a few minor changes. As before, the tropical curve $\Gamma$ is the dual graph of a prestable curve with infinite legs. We will be interested in piecewise linear maps
\[
\Gamma \to \mathbb R_{\geq 0}^r. 
\]
We will be interested in dual graphs of genus $g$, with vertices decorated by effective curve classes $\beta_V$ with total degree $\beta$. 

The contact orders impose additional conditions. For each marked point $p_i$ we have a column of the contact order matrix $[c_{ij}]$. This in turn determines a vector in $ \mathbb R_{\geq 0}^r$ -- take the weighted sum of the corresponding ray generators. This vector can be viewed as a weight times a primitive vector. We demand that the map sends the half edge corresponding to $p_i$ onto a line in the directions of this primitive vector, and with slope given by the weight. 

As before there is a cone stack $T_\Lambda(\RR_{\geq 0}^r)$. It is not of finite type, but by geometric considerations, including the balancing condition~\cite{GS13}, one can pick out a finite type subcomplex. 

\subsection{Evaluation structure}

Any of the three spaces -- cheap, expanded, or logarithmic -- are sufficient for the degeneration formula. The next step is to describe the ingredients required for the ``gluing condition'' in the degeneration. This comes from evaluation maps. 

Fix discrete data $\Lambda = (g,n,\beta,[c_{ij}])$ and a pair $(X|\partial X)$.

For each marked point $p_i$, we associate a stratum
\[
\mathsf{Ev}_{\Lambda,p_i} \subset X
\]
defined as follows. If $c_{ij} > 0$ for some $j$, then
\[
\mathsf{Ev}_{\Lambda,p_i} := \bigcap_{\,j \;:\; c_{ij} > 0} D_j,
\]
the intersection of those boundary components along which the contact order at $p_i$ is positive. If $c_{ij}=0$ for all $j$, we set
\[
\mathsf{Ev}_{\Lambda,p_i} := X.
\]

The logarithmic evaluation maps determine a morphism
\[
\mathsf{ev} \colon 
\Mbar^{\log}_{\Lambda}(X|\partial X)
\;\longrightarrow\;
\prod_{i=1}^n \mathsf{Ev}_{\Lambda,p_i},
\]
whose $i$-th component records the image of the $i$-th marked point.
Analogous evaluation morphisms are defined in the expanded theory. With either theory, we can define invariants.

\begin{definition}[Log GW invariants/cycles]
Let $\gamma_i \in \mathsf{CH}^\star(\mathsf{Ev}_{\Lambda,p_i})$.
The \underline{logarithmic Gromov–Witten invariant} of $(X|\partial X)$ with discrete data $\Lambda$ and insertions $\gamma_1,\dots,\gamma_n$ is:
\[
\big\langle \gamma_1,\dots,\gamma_n \big\rangle^{X|\partial X}_{\Lambda}
:=
\deg \left(
\left( \prod_{i=1}^n \mathsf{ev}_i^\star(\gamma_i) \right)
\cap 
\big[ \Mbar^{\sf log}_{\Lambda}(X|\partial X) \big]^{\mathrm{vir}}
\right).
\]

More generally, if
\[
\pi\colon \Mbar^{\sf log}_{\Lambda}(X|\partial X) \;\longrightarrow\; \Mbar_{g,n}
\]
is the morphism to the moduli space of stable curves, we define \underline{logarithmic Gromov–Witten cycles} by pull/push:
\[
\pi_\star\!\left(
\left( \prod_{i=1}^n \mathsf{ev}_i^\star(\gamma_i) \right)
\cap 
\big[ \Mbar^{\sf log}_{\Lambda}(X|\partial X) \big]^{\mathrm{vir}}
\right)\in \mathsf{CH}^\star(\Mbar_{g,n};\QQ).
\]
Analogous cycles can be defined in (co)homology. 
\end{definition}

\subsection{The logarithmic degeneration formula}\label{sec: deg-formula}

A key ingredient that goes into the degeneration formula is a refined form of intersection theory for pairs, colloquially called \underline{logarithmic intersection theory}. To explain its need, we start with a naive, special case of the degeneration formula that is typically false.

\subsubsection{A false formula}

Let $\mathcal X\to B$ be a projective snc degeneration and curve class $\beta$. Recall that there is a cone stack
\[
\mathsf T_{g,n}(\mathbb R_{\geq 0}/\mathbb R_{\geq 0})\to \RR_{\geq 0}
\]
of vertical tropical maps. By work of Abramovich--Chen--Gross--Siebert~\cite{ACGS15}, in what is called the ``decomposition step'' of the degeneration formula, the special fiber of the mapping space $\Mbar^{\sf log}_{g,n}(\mathcal X_0,\beta)$ is broken into a union:
\[
\Mbar^{\sf log}_{g,n}(\mathcal X_0,\beta) = \bigcup_\gamma \Mbar^{\sf log}_{\gamma}(\mathcal X_0,\beta)
\]
where $\gamma$ is what is called a \underline{rigid tropical map}. In the language we have, this is a combinatorial type of map corresponding to a $1$-dimensional cone of $\mathsf T_{g,n}(\mathbb R_{\geq 0}/\mathbb R_{\geq 0})$ that surjects onto the base. 

Each of the spaces $\Mbar^{\sf log}_{\gamma}(\mathcal X_0,\beta)$ is a \underline{virtual irreducible component} of the special fiber space $\Mbar^{\sf log}_{g,n}(\mathcal X_0,\beta)$. What this really means is that at the level of universal targets,
\[
\mathfrak M^{\sf log}_{g,n}(\mathcal A_{B,0})
\]
each of these $\gamma$ indexes a genuine irreducible component. The map 
\[
\Mbar^{\sf log}_{g,n}(\mathcal X_0,\beta)\to \mathfrak M^{\sf log}_{g,n}(\mathcal A_{B,0})
\]
is equipped with a virtual structure as we have discussed. So by pulling back the irreducible components and their fundamental classes, we get the pieces of the decomposition above, together with natural virtual classes. Note that the intrinsic scheme theoretic multiplicity of these components is encoded by the slope of the map from the ray of $\mathsf T_{g,n}(\mathbb R_{\geq 0}/\mathbb R_{\geq 0})\to \RR_{\geq 0}$. 

In particular, fixing $\gamma$ fixes the topological type of a nodal curve, the strata of $\mathcal X_0$ to which they map, their tangency orders, and the curve classes. In order to see the issues, it suffices to add a simplifying assumption\footnote{If this is not true, we can replace $\mathcal X$ with a blowup to ensure the condition at the outset. This turns out not to change the invariants.}: assume each vertex of $\gamma$ is assigned to an irreducible component (rather than a deeper stratum) of $\mathcal X_0$. At the tropical level, this means that the tropical map along the chosen ray sends vertices of the domain tropical curve to vertices of the target simplex.

The rigid map $\gamma$ can have automorphisms -- automorphisms of the domain graph that preserves the labels and commutes with the map. By passing to a finite cover that labels the edges, we can remove the autormosphisms. Now, for each vertex $V$ (of the source graph of) the rigid map $\gamma$, we can associated a target pair $(X_V|\partial X_V)$ and discrete data $\Lambda_V$.

For each edge $E$ incident to $V$ the moduli space $\Mbar^{\sf log}_{\Lambda_V}(X_V|\partial X_V)$ has an evaluation map
\[
\Mbar^{\sf log}_{\Lambda_V}(X_V|\partial X_V)\to \mathsf{Ev}_{\Lambda_V,E}.
\]
Ranging over all vertices, noticing that each edge shows up twice, and writing $\mathsf{Ev}_{\Lambda_V,E} = \mathsf{Ev}_{E}$ we have maps:
\[
\mathsf{ev}_\gamma\colon \prod_V \Mbar^{\sf log}_{\Lambda_V}(X_V|\partial X_V)\to \prod_E \mathsf{Ev}^2_{E}.
\]

The simplest guess for our gluing condition is to impose the diagonal condition:
\[
\mathsf{ev}_\gamma^\star(\Delta_{\prod \mathsf{Ev}^2_{E}})\cap [\Mbar^{\sf log}_{\Lambda_V}(X_V|\partial X_V)]^{\sf vir}.
\]
There are some combinatorial factors attached to $\gamma$, but we would like to say that, for example after pushing forward to the stack of all maps $\Mbar_{g,n}(\mathcal X_0,\beta)$, that the classes
\[
\mathsf{ev}_\gamma^\star\big(\Delta_{\prod \mathsf{Ev}^2_{E}}\big)\cap 
\big[\Mbar^{\sf log}_{\Lambda_V}(X_V|\partial X_V)\big]^{\sf vir}
\quad\text{and}\quad 
\big[\Mbar_{\gamma}^{\sf log}(\mathcal X_0,\beta)\big]^{\sf vir}.
\]
are proportional, in fact equal up to these bookkeeping factors. This is exactly the statement of Jun Li's degeneration formula in the double point degeneration setting. \underline{However, this is typically false} for simple normal crossings degenerations. 

In fact, the formula fails even at the level of sets -- the evaluation pullback is too big and contains the correct space $\Mbar_\gamma^{\sf log}(\mathcal X_0,\beta)$ as a proper closed subspace. 

The formula is true ``on the interior'', meaning when the logarithmic structure is trivial, but it fails on the boundary, and is essentially an issue of excess intersection theory. The basic issue has to do with the tropical stratifications. The map $\mathsf{ev}_\gamma$ induces a map on stratifications, as all logarithmic maps do, but this induced map is ``nontransverse'' with respect to the stratification.

\subsubsection{Logarithmic enhancements}

To explain the fix to the false formula above, we use the space of expanded maps. This is just an expository preference well-suited to our narrative, and because it is the setting in which the degeneration formula~\cite{MR23} is actually proved. 

We fix a target $(X|\partial X)$, discrete data $\Lambda$, and an evaluation space, taken over all marked points, $\mathsf{Ev}_\Lambda$. We impose an additional condition that in the matrix $c_{ij}$ for each marked point $p_i$ there is at most one $j$ such that $c_{ij}$ is nonzero. This condition is not restrictive -- we can achieve it by blowing up $(X|\partial X)$ along strata. By work of Abramovich--Wise in the logarithmic setting this does not change the GW theory~\cite{AW}. 

For each marked point $p_i$, the stratum $\mathsf{Ev}_{\Lambda,p_i}$ itself has a natural snc divisor -- its intersection with the divisors that do not contain it. The space $\mathsf{Ev}_{\Lambda}$ is a product of these strata, and so also has a natural simple normal crossings divisor. 

Suppose
\[
\mathsf{Ev}_\Lambda^\dagger\longrightarrow\mathsf{Ev}_\Lambda
\]
is a strata blowup. We also have on evaluation map $\Mbar^{\sf exp}_\Lambda(X|\partial X)\to \mathsf{Ev}_\Lambda$. We now describe a lift that will fit into a commutative diagram:
\[
\begin{tikzcd}
\boxed{?} \arrow{d}\arrow{r}& \mathsf{Ev}_\Lambda^\dagger\arrow{d}\\
\Mbar^{\sf exp}_\Lambda(X|\partial X)\arrow{r}&\mathsf{Ev}_\Lambda
\end{tikzcd}
\]
such that push forward along the left vertical identifies virtual classes. Note that this will \underline{not} be the pullback/total transform. Rather, it behaves more like a \underline{strict} transform. 

This is where the flexibility of having a choice of stack of expansion is useful. Recall that the stack of expansions came with a choice cone complex structure on a certain topological space, and the set of such admissible cone structures is a partially ordered set under refinement. If we refine the stack of expansions, we can still make the mapping stack construction -- every stack of expansions has an associated mapping space. 

\begin{theorem}[Enhanced logarithmic GW theory]
For a fixed blowup $\mathsf{Ev}^\dagger\to \mathsf{Ev}$, any sufficiently refined stack of expansion $\mathsf{Exp}^\dagger(X|\partial X)\to \mathsf{Exp}(X|\partial X)$, the associated moduli stack of expanded stable maps $\Mbar^{\sf exp}_\Lambda(X|\partial X)^\dagger$ fits into a commutative diagram
\[
\begin{tikzcd}
\Mbar^{\sf exp}_\Lambda(X|\partial X)^\dagger \arrow{d}\arrow{r}& \mathsf{Ev}_\Lambda^\dagger\arrow{d}\\
\Mbar^{\sf exp}_\Lambda(X|\partial X)\arrow{r}&\mathsf{Ev}_\Lambda.
\end{tikzcd}
\]
The left vertical map identifies virtual fundamental classes.

Logarithmic GW cycles and invariants are defined in the standard way with insertions in $\mathsf{CH}^\star(\mathsf{Ev}^\dagger_\Lambda)$ or $\mathsf{H}^\star(\mathsf{Ev}^\dagger_\Lambda)$, by pulling back to $\Mbar^{\sf exp}_\Lambda(X|\partial X)^\dagger$ and pairing with the virtual class.
\end{theorem}

We refer to the top horizontal arrow in the square above as a \underline{logarithmic lift}. 

A few key points should be highlighted here. \\

\noindent
\underline{Independence of choices.} If we take a further blowup $\mathsf{Exp}(X|\partial X)^\dagger\to \mathsf{Exp}(X|\partial X)^\ddagger$, the associated map
\[
\Mbar^{\sf exp}_\Lambda(X|\partial X)^\ddagger\to \Mbar^{\sf exp}_\Lambda(X|\partial X)^\dagger
\]
also identifies virtual classes. In particular, intersection products of a cohomology class in some blowup $\mathsf{Ev}^\dagger$ of the evaluation space with the virtual class of a compatible mapping space are independent of the choice of mapping space (after appropriate pushforward). \\

\noindent
\underline{Exotic insertions.} First, the blowup
\[
\mathsf{Ev}^\dagger_\Lambda \to \prod_i \mathsf{Ev}_{\Lambda,p_i}
\]
is a blowup of a product but is \underline{not a product of blowups}. A class $\gamma$ in $\mathsf{CH}^\star(\mathsf{Ev}^\dagger)$ that is \underline{not} pulled back along the blowup is called an \underline{exotic cohomology class}, and the corresponding invariants are \underline{exotic log GW invariants}. Because these blowups are ``criss-cross'', the associated logarithmic GW theory (cycles and invariants) are not detected (at least not formally) by the GW theory of strata blowups of $(X|\partial X)$ itself.\\

\noindent
\underline{When have you blown up enough?} By the discussion above, we have created a significantly larger class of logarithmic GW invariants/cycles. We can think of logarithmic GW theory as being operators on a space of insertions, the latter being an appropriate tensor power of the cohomology of the target. 

But since we can now use any cohomology class in any strata blowup of $\mathsf{Ev}_\Lambda$ the ``space of insertions'' is really the following direct limit:
\[
\mathsf{CH}^\star_{\sf log}(X|\partial X) = \varinjlim_{X'\to X} \mathsf{CH}^\star(X'), \ \ X'\to X \textnormal{ a strata blowup},
\]
where the transition maps are given by pullback. 

It is natural to ask whether this is genuinely an infinite amount of data. In fact it isn't. Once $\mathsf{Ev}^\dagger_\Lambda$ is a sufficiently refined blowup it will satisfy the following property. For a blowup $\mathsf{Ev}^\ddagger_\Lambda\to\mathsf{Ev}^\dagger_\Lambda$, a compatible lift
\[
\begin{tikzcd}
\Mbar^{\sf exp}_\Lambda(X|\partial X)^\ddagger \arrow[swap]{d}{p}\arrow{r}& \mathsf{Ev}_\Lambda^\ddagger\arrow{d}{q}\\
\Mbar^{\sf exp}_\Lambda(X|\partial X)^\dagger\arrow{r}&\mathsf{Ev}^\dagger_\Lambda.
\end{tikzcd}
\]
and a cohomology class $\gamma$ on $\mathsf{Ev}^\ddagger_\Lambda$, we have an equality
\[
p_\star\left([\Mbar^{\sf exp}_\Lambda(X|\partial X)^\ddagger]^{\sf vir} \cap \gamma \right) = [\Mbar^{\sf exp}_\Lambda(X|\partial X)^\dagger]^{\sf vir} \cap q_\star \gamma,
\]
where we have slightly abused notation by suppressing the pullbacks to the moduli space. In particular, there are no new invariants after this point. This equality might fail for a \underline{particular} choice of $\mathsf{Ev}_\Lambda^\dagger$, and this often happens for the initial choice of $\mathsf{Ev}_\Lambda$. But after enough blowups, things stabilize.

In fact, there is a purely tropical criterion for when we have blown up sufficiently. The evaluation map (before any blowup)
\[
\mathsf{ev}\colon \Mbar_\Lambda^{\sf exp}(X|\partial X)\to \mathsf{Ev}_\Lambda
\]
has a tropical incarnation. Specifically, we can let $T_\Lambda(\Sigma_X)$ denote the set of maps from tropical curves to the cone complex associated to $(X|\partial X)$\footnote{Earlier in the text we were using $\mathbb R_{\geq 0}^r$ as a proxy for the cone complex $\Sigma_X$. If we put the rays in bijection with the components of $\partial X$, the true cone complex $\Sigma_X$ is the subcomplex spanned by those subsets of rays whose corresponding divisors have nonempty intersection in $X$.}. More precisely, we should take a finite type subcomplex large enough to contain all the combinatorial types of all maps in $\Mbar_\Lambda^{\sf exp}(X|\partial X)$. For example, we can take the subcomplex obtained by imposing degree decorations on the vertices and the balancing condition. 

The evaluation map induces a map of cone complexes
\[
T_\Lambda(\Sigma_X)\to \Sigma(\mathsf{Ev}_\Lambda).
\]
This map can be understood in a very concrete way. By the hypotheses on the tropical map and the contact order matrix, each half-ray corresponding to a marked point $p_i$ maps parallel to one of the rays of $\Sigma_X$. If $c_{ij}$ is positive, then the ray corresponding to $p_i$ maps parallel to the ray $\rho_j$ corresponding to $D_j$. The fan of $D_j$ is the star of the ray $\rho_j$. See Figure~\ref{fig: trop-eval}.
\begin{figure}[h!]
\begin{center}
\includegraphics[scale=0.6]{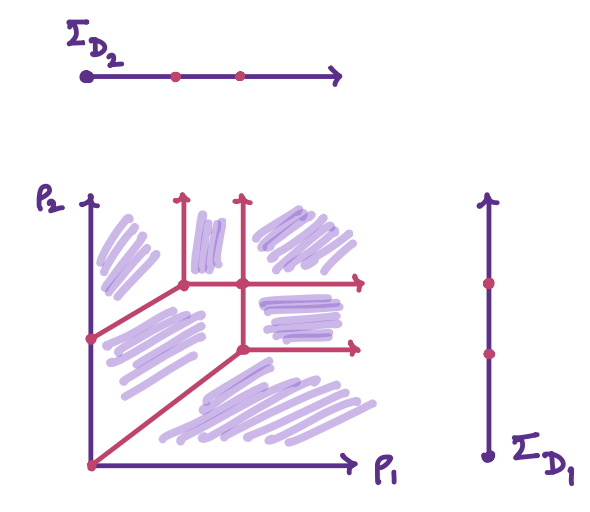}
\end{center}
\caption{The tropicalization of the evaluation map sends a point represented by the tropical map shown to the points on the rays, represented as points on the ray above and to the right in the picture.}\label{fig: trop-eval}
\end{figure}

For the next definition, we remind the reader that the category of cone complexes, together with morphisms between them, admits fiber products, constructed essentially in the same way as common refinements.

\begin{definition}
A subdivision $\Sigma(\mathsf{Ev}_\Lambda)^\dagger\to \Sigma(\mathsf{Ev}_\Lambda)$ is a \underline{flattening} of $T_\Lambda(\Sigma_X)\to \Sigma(\mathsf{Ev}_\Lambda)$ if there exists a subdivision $T_\Lambda(\Sigma_X)^\dagger\to T_\Lambda(\Sigma_X)$ such that the map induced by pullback:
\[
T_\Lambda(\Sigma_X)^\dagger \to \Sigma(\mathsf{Ev}_\Lambda)^\dagger
\] 
has the following property: every cone of the domain maps surjectively onto a cone of the target. 
\end{definition}

Such flattenings always exist by the toroidal semistable reduction theorem, see~\cite{AK00,Mol16}. 

\subsubsection{The statement of the degeneration formula}

We now arrive at the statement of the degeneration formula. Let $\mathcal X \to B$ be a simple normal crossings degeneration with discrete data $\beta$. We consider the geometric and tropical spaces
\[
\Mbar^{\sf exp}_{g,n}(\mathcal X_0,\beta) 
\quad \textnormal{and} \quad 
\mathsf T^{\sf exp}_{g,n}(\RR_{\geq 0}^r/\RR_{\geq 0}) \to \RR_{\geq 0}.
\]
Among the rays of $\mathsf T^{\sf exp}_{g,n}(\RR_{\geq 0}^r/\RR_{\geq 0})$, we single out those that map surjectively to $\RR_{\geq 0}$. Their combinatorial types are precisely the \underline{rigid tropical maps}.

After rigidifying the automorphisms of such a rigid map $\gamma$, we can extract discrete data at each vertex. We denote the automorphism group of $\gamma$ by $\mathsf{Aut}(\gamma)$ and the associated discrete data, to be described below, by $\Lambda_V$:

\begin{enumerate}[(i)]
\item In the expanded setting, a rigid tropical curve includes the data of a $1$-complex of the simplex, and hence possibly a further expansion of $\mathcal X_0$. 
\item Each vertex $V$ of the source graph maps to a vertex of this $1$-complex, and therefore this determines a target pair $(X_V|\partial X_V)$. 
\item The half-edges corresponding to the marked points, together with the edges incident to a given vertex $V$, form the indexing set for the marked points. Since the edges are not labeled, one must choose a labeling to rigidify the combinatorial type; this choice produces an automorphism factor later. 
\item The directions of the edges emanating from each vertex determine a contact order matrix at that vertex. Given an edge $E$, it has a well-defined slope under the map, which we denote by $m_E$; this is the contact order of the edge with the corresponding divisor. 
\end{enumerate}

One more important point that we note is the following. The product of evaluation spaces over all vertices can be written both as a product over vertices and a product over edges:
\[
\prod_V \mathsf{Ev}_{\Lambda_V} = \prod_E \mathsf{Ev}_E^2.
\]
The description on the left is more natural, but the description on the right determines a key cycle, namely the diagonal.

With this understood, we can state the degeneration formula. 

\begin{theorem}[Logarithmic degeneration formula]\label{thm: deg-form}
For each rigid tropical curve and for each vertex in its domain graph, let
\[
\Mbar^{\sf exp}_{\Lambda_V}(X_V|\partial X_V)^\dagger\to \mathsf{Ev}_{\Lambda_V}^\dagger
\]
be a flattening blowup of the evaluation map. Let $\delta_\dagger\colon \Delta^\dagger\hookrightarrow \prod_V \mathsf{Ev}_{\Lambda_V}$ be the strict transform of the diagonal locus. Then we have:
\[
[\Mbar^{\sf exp}_{g,n}(\mathcal X_0,\beta)]^{\sf vir} = \sum_\gamma \; \frac{\prod_E m_E}{\mathsf{Aut}(\gamma)}\; \delta_\dagger^! \left( \prod_V \left[\Mbar^{\sf exp}_{\Lambda_V}(X_V|\partial X_V)^\dagger\right]^{\sf vir} \right),
\]
with the implicit pushforward of the right hand side to $\Mbar^{\sf exp}_{g,n}(\mathcal X_0,\beta)$ under the natural projection. 
\end{theorem}

In Betti cohomology, the diagonal $[\Delta^\dagger]\in \prod_V \mathsf{Ev}_{\Lambda_V}^\dagger$ admits a K\"unneth decomposition. The K\"unneth factors are naturally \underline{exotic} cohomology classes on $\mathsf{Ev}_{\Lambda_V}^\dagger$, so we can write
\[
[\Delta^\dagger] = \sum_j \otimes_V \delta_V^{(j)}\in \bigotimes H^\star(\mathsf{Ev}^\dagger_{\Lambda_V}).
\]
In particular, the result gives an explicit \underline{numerical} formula -- an explicit form is stated in~\cite[Section~8]{MR23}\footnote{Note that there is a typographical error in the formula stated there: the alternating sign should be removed.}. 

\begin{remark}[Complexity and special cases]
The part of the formula that may appear mysterious is the construction of the flattening blowup. But we emphasize that it is a question entirely about understanding the moduli space of tropical curves and its evaluation map in detail. While there may be other formalisms to approach this question, the combinatorics involved appears to be an inherent complexity in the problem. 

In some cases, the complexity of the flattening goes away. For example, of $\mathcal X_0$ has at worst triple points (e.g. in dimension $2$) and when the genus is $0$, or at least all dual graphs are $0$, the flattening blowup step is essentially trivial, see e.g.~\cite[Section~6]{R19}. See also~\cite{Bou17,KHSUK} for some remarkable uses of this special case. 
\end{remark}

\begin{corollary}[General fiber invariants from log invariants]
Given $\mathcal X\to B$, the GW invariants of $\mathcal X_\eta$ with insertions pulled back from the cohomology of $\mathcal X$ can be expressed as a linear combination of \underline{exotic} logarithmic GW invariants of the pairs $(X_i|\partial X_i)$, where the $X_i$ range over the components of the degenerations determined by all rigid tropical curves. 
\end{corollary}

We discuss refinements of this result in the applications section. But we flag one of them -- exotic invariants can always be re-expressed in terms of non-exotic invariants and correction terms that are ``lower order''. If one is willing to forego some explicitness, in fact the exotic insertions can be removed. 

\section{Examples}

In light of the degeneration formula, the space of logarithmic maps to $\mathcal X_0$ is built out of the space of maps to pairs $(X|\partial X)$. We now share a series of examples of such mapping spaces. These are the examples that have helped the author gain a footing in these topics, and the hope is that they will do the same for the reader. 

\subsection{One parameter subgroups}
This example is the subject of beautiful paper by Chen--Satriano~\cite{CS12}. Let $(X|\partial X)$ be a pair consisting of a smooth toric variety and its torus invariant boundary divisor, and say $T$ is the dense torus. Let $v$ a vector in the cocharacter lattice of the torus. It determines a $1$-parameter subgroup $H_v$ and so a map $\mathbb G_m\to X$ that extends to a map of pairs:
\[
\varphi_v\colon(\mathbb P^1|0,\infty)\to (X|\partial X).
\]
Since the pullback of $\partial X$ is $\{0,\infty\}$ this is a logarithmic map. Let $\Lambda$ be the discrete data of this map. Composition with the $T$-action gives more logarithmic maps, but since maps are considered up to automorphism, the subtorus $H_v$ acts trivially. This gives an embedding
\[
T/H_v\hookrightarrow \Mbar^{\sf log}_\Lambda(X|\partial X).
\]
By a deformation theory argument, one sees that $\Mbar^{\sf log}_\Lambda(X|\partial X)$ is logarithmically smooth of the expected dimension, which is $\dim X-1$. The locus where the logarithmic structure is trivial is dense. In fact, any such map differs from $\varphi_v$ by the action of $T$, so $T/H_v$ is dense and open. 

Chen and Satriano show that $\Mbar^{\sf log}_\Lambda(X|\partial X)$ is itself a toric variety, and coincides with the famous \underline{Chow quotient} construction. In particular, there is a completely elementary procedure to describe the fan of $\Mbar^{\sf log}_\Lambda(X|\partial X)$ starting with $\Sigma_X$ and the vector $v$. See~\cite{CS12,KSZ91} as well as~\cite{AM14} for a higher dimensional generalization. 

\subsection{Rational curves in toric varieties}

The next simplest examples concern rational curves in toric varieties with more general tangency. The following is adapted from~\cite{R15b,RW19}. 

As before, let $(X|\partial X)$ be a pair consisting of a smooth toric variety and its torus invariant boundary divisor, and say $T$ is the dense torus. Fix $g = 0$ and a contact order matrix $[c_{ij}]$. Note that this determines the curve class. We describe $\Mbar_\Lambda(X|\partial X)$. Again, the space is unobstructed of the expected dimension; let $\mathsf M_\Lambda(X|\partial X)$ denote the locus where the curve is smooth and maps generically to $T$. On this locus, if $C$ denotes the domain curve, the rational map
\[
C\dashrightarrow T
\]
has fixed orders of zeroes and poles at the $p_i$ and is defined away from here. A rational function all of whose zeros and poles (and their orders) are known is unique up to scalar. This means that
\[
\mathsf M_\Lambda(X|\partial X)\to \mathsf M_{0,n}
\]
is a torsor under $T$. Since the Picard group of $\mathsf M_{0,n}$ is trivial (e.g. by excision) this means that
\[
\mathsf M_\Lambda(X|\partial X)\cong \mathsf M_{0,n} \times T.
\]
Note if one of the $p_i$ has trivial tangency order, one can use the evaluation map to produce a trivialization of the torsor.  

Over the boundary, this identification degenerates, but nevertheless, there is an isomorphism
\[
\Mbar^{\sf log}_\Lambda(X|\partial X)\cong (\Mbar_{0,n} \times X)^\dagger
\]
with a blowup of a product. The blowup is specified by an explicit combinatorial refinement of the cone complex associated to $\Mbar_{0,n} \times X$. 

A pleasant special case of the above example is when the toric variety is equal to $\mathbb P^r$ and the degree is equal to $1$. In this case, one gets an interesting birational model of the Grassmannian $Gr(2,r+1)$, closely related to the geometry of Chow quotients~\cite{Kap93}. 

\begin{remark}[Higher genus and double ramification]
The essential point above is that degree $0$ divisors with prescribed zeroes and poles (given by the contact orders) are principal and admit a unique section up to scalar. In higher genus, one may still consider those divisors that are trivial. Over the moduli space of smooth curves, this condition can be expressed as the intersection of two copies of $\mathsf M_{g,n}$ inside the universal Picard stack: one given by the Abel--Jacobi section determined by the contact orders, and the other by the zero section. 

For nodal curves, one can attempt to extend this picture using either compactified Jacobians or the logarithmic Picard group~\cite{AP21,Hol17,MW17,MW18}. This leads to an approach to higher genus logarithmic GW theory for toric pairs and to a connection with \underline{higher double ramification cycles}; see~\cite{HMPPS,JPPZ,MR21,RUK22}. 
\end{remark}

\subsection{Relative product formulas}

The following example comes from the thesis of Nabijou~\cite{Nab19} and the article~\cite{NR19}. 

Let $H_1,H_2\subset\mathbb P^r$ be two distinct hyperplanes. Fix $g = 0$ and degree $d$. Let $[c_{ij}]$ be the $n\times 2$ matrix of contact orders -- the row sums must be $d$. We can consider three pairs:
\[
(\mathbb P^r|H_1), \; (\mathbb P^r|H_2), \; (\mathbb P^r|H_1+H_2).
\]
Let $\Lambda_1,\Lambda_2,\Lambda$ be the three discrete data sets, with contact orders $[c_{i1}]$, $[c_{i2}]$, and $[c_{ij}]$. There is a commutative diagram of spaces:
\[
\begin{tikzcd}
\Mbar^{\sf log}_\Lambda(\mathbb P^r|H_1+H_2)\arrow{d}\arrow{r} & \Mbar^{\sf log}_\Lambda(\mathbb P^r|H_1)\arrow{d}\\
\Mbar^{\sf log}_\Lambda(\mathbb P^r|H_2)\arrow{r} & \Mbar_{0,n}(\mathbb P^r,d). 
\end{tikzcd}
\] 
Note that each of these spaces is irreducible of the expected dimension -- this follows from deformation theory and explicit parameterizations of the locus of non-degenerate maps in each case. However, Nabijou observes that, if we let $\overline{\mathsf{N}}^{\sf log}_\Lambda(\mathbb P^r|H_1+H_2)$ denote the fiber product, the map
\[
\Mbar^{\sf log}_\Lambda(\mathbb P^r|H_1+H_2)\to\overline{\mathsf{N}}^{\sf log}_\Lambda(\mathbb P^r|H_1+H_2)
\]
typically \underline{fails} to be an isomorphism. The failure is almost immediate, as shown in~\cite[Section~2]{NR19}. Take $r= 2$, $n = 2$ (so two marked points) and degree $d = 2$. Choose the tangency order with $H_i$ to be $2$ at $p_i$. The locus of maps onto smooth conics is $3$-dimensional, and its closure is a component of $\overline{\mathsf{N}}^{\sf log}_\Lambda(\mathbb P^r|H_1+H_2)$. Another $3$-dimensional component parameterizes maps
\[
C_0\cup C_1\cup C_2\to \PP^2
\]
where $C_0$ carries the marked points, is contracted to $H_1\cap H_2$, and meets $C_1$ and $C_2$ at distinct points. 

In this situation, the logarithmic space picks out one of the components of the intersection. The underlying geometry is actually very simple. On the locus of maps from smooth domains, each space $\mathsf M_\Lambda(\mathbb P^r|H_i)$ embeds in $\Mbar_{0,n}(\mathbb P^r,d)$. The true logarithmic space is the (normalization of the) closure of their intersection. The naive space is the fiber product of (the normalizations of) their closures.

Of course, the two theories can be compared -- by judicious use of blowups and excess intersection theory, one can express the difference between them in terms of natural tautological classes. The smooth pair theory $(\mathbb P^r|H)$ is understood through beautiful work of Gathmann~\cite{Gat02} and Vakil~\cite{Vak00}, and this leads to a method for computing logarithmic GW invariants in this case.

The example above is related to an interesting link between logarithmic structures and orbifold Gromov--Witten theory, see~\cite{ACW,BNR22,TY23}. For a discussion of similar phenomena from the perspective of logarithmic intersection theory, see~\cite{Herr}.

\subsection{Conics through points via tropical geometry}

A great deal of intuition for how the degeneration formula should work comes from the tropical correspondence theorems~\cite{Mi03,NS06}, which treat the case when $X$ is toric and $\partial X$ is the full toric boundary. We explain this, in the language developed here, in the simplest nontrivial case -- counting conics through $5$ points. 

The setup is slightly different from what we have considered so far. Specifically, we have not yet considered degenerations $\mathcal X\to B$ whose general fiber is itself a pair $(\mathcal X_\eta|\partial \mathcal X_\eta)$, but both the theory and the degeneration formula extend to this case. 

To set this up, we start with the fan $\Sigma_{\mathbb P^2|\partial \mathbb P^2}$ of the projective plane, with cocharacter lattice $N$. Choose $5$ generic integer points in $N$ and let $\mathcal P$ be any polyhedral decomposition with integer vertices satisfying the following three properties: (i) the $5$ points are vertices of $\mathcal P$, (ii) the unbounded rays of $\mathcal P$ are each parallel to one of the rays of $\Sigma_{\mathbb P^2|\partial \mathbb P^2}$, and (iii) the cone over $\mathcal P$ is a smooth fan. This produces a degeneration of $\mathbb P^2$ over $\mathbb A^1$. The general fiber is $(\mathbb P^2|\partial \mathbb P^2)$. The special fiber has $5$ distinguished components, chosen in advance. 

The constructions discussed above can be adapted to produce a degeneration $\Mbar_{0,11}^{\sf exp}(\mathcal X/B,2)$. The general fiber is the space of conics mapping to $(\mathbb P^2|\partial \mathbb P^2)$ -- among the $11$ marked points, $6$ record the intersections with the toric boundary (two for each boundary divisor), and $5$ have tangency order $0$. 

We now impose the $5$ point conditions. To do this, we choose fiberwise point conditions, i.e.\ sections of $\mathcal X\to B$. We choose sections $s_1,\ldots, s_5$ whose special fibers are generic points on the distinguished components. We then consider the moduli space of maps $\Mbar_{0,11}^{\sf exp}(\mathcal X/B,2)^{\sf pts}$ that pass through these $5$ points in the fibers. Of course, we know that the general fiber consists of a single point (up to labeling of the marked points), but let us ignore this for the moment. 

The special fiber has components indexed by tropical curves, and straightforward bookkeeping shows that these tropical curves must pass through the $5$ chosen distinguished points. A purely combinatorial argument, involving the balancing condition, reveals that there is in fact a single relevant tropical curve, shown below in Figure~\ref{fig: conic}.
\begin{figure}[h!]
\begin{center}
\includegraphics[scale=0.4]{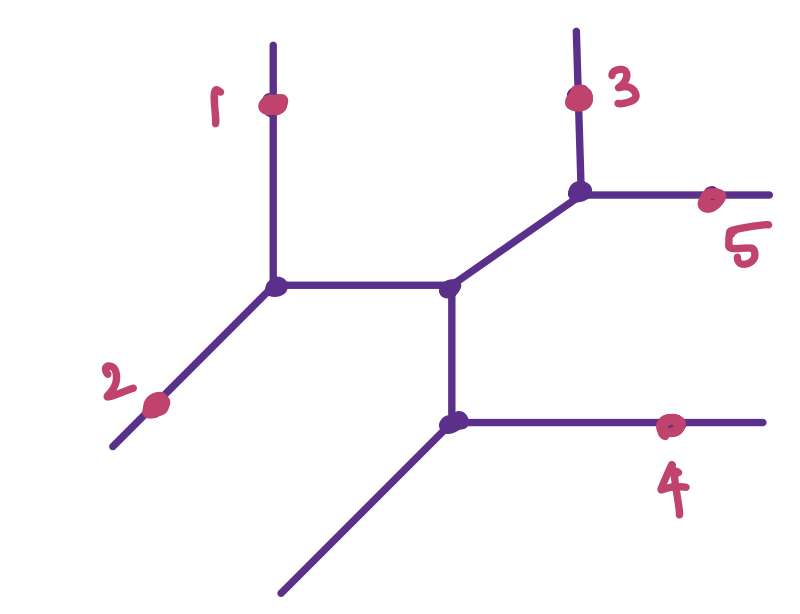}
\end{center}
\caption{A tropical conic, encoding a degeneration of $\mathbb P^2$.}\label{fig: conic}
\end{figure}
Each rigid tropical curve determines a moduli space, and in this case these moduli spaces are of the type described above -- rational curves in toric surfaces. The possibilities are either lines in $\mathbb P^2$ or a fiber class in $\mathbb P^1\times\mathbb P^1$. The degeneration formula now implies (in a rather complicated way) that there is a conic through the $5$ points. 

One can attempt to implement this strategy more generally, leading to the famous \underline{tropical correspondence theorems} of Mikhalkin~\cite{Mi03} and Nishinou--Siebert~\cite{NS06}. The former treats toric surfaces with point conditions in arbitrary genus; the geometric idea is essentially the same, but carrying it out requires a much more delicate analysis of multiplicities. The latter treats rational curves in toric varieties, again with additional work but in a similar spirit to what we have suggested above. There are further generalizations of these ideas -- to descendant invariants~\cite{MR16}, refined curve counts~\cite{Bou17}, and all genus counts in threefolds~\cite{MR26,Par17}.

\section{GW theory and tautological classes}

The logarithmic degeneration formula serves as a starting point for several reconstruction results in GW theory, developed in~\cite{MR25}. We briefly discuss some of these results.

\subsection{Logarithmic/absolute}

We have introduced a ``new'' Gromov--Witten theory for pairs $(X|\partial X)$. A natural question is to what extent the resulting invariants and cycles (i.e.\ their pushforwards to $\Mbar_{g,n}$) are genuinely new.

The following result is proved in~\cite{MR25}. 

\begin{theorem}[Logarithmic in terms of absolute]
Let $(X|\partial X)$ be a simple normal crossings pair. The cohomological logarithmic GW cycles (resp.\ invariants), possibly with exotic insertions, of $(X|\partial X)$ are determined by the ordinary GW cycles (resp.\ invariants) of its strata. 
\end{theorem}

The proof relies on several ingredients, including a study of the logarithmic GW cycles of projectivized split vector bundles in terms of the cycles of the base, and an algorithm expressing exotic insertions in terms of non-exotic ones. In the smooth pair case, the result reduces to ~\cite{MP06}. See also the reconstruction result in~\cite{BNR22}. 

One interpretation of the theorem is that logarithmic theory introduces no new unknowns into GW theory. Rather, the degeneration formula produces new relations among the existing ones; taken together, this makes it an effective computational tool.

\subsection{GW cycles in cohomology}

Recall that the cohomology ring of the moduli space of curves $\mathsf H^\star(\Mbar_{g,n};\QQ)$ contains a distinguished subring $\mathsf{RH}^\star(\Mbar_{g,n})$, the \underline{tautological ring}. It is the smallest system of subrings, defined simultaneously for all $(g,n)$, that is closed under pushforward along forgetful and gluing morphisms. 

\begin{theorem}[General fiber in terms of special fiber]
Let $\mathcal X\to B$ be a simple normal crossings degeneration. Let $\mathcal X_b$ be a smooth fiber and $\mathcal X_0$ the special fiber. The cohomological GW cycles of $\mathcal X_b$, with arbitrary insertions, lie in the $\mathsf{RH}^\star(\Mbar_{g,n})$-linear span of the ordinary GW cycles of the strata of $\mathcal X_0$. 
\end{theorem}

Motivated by their work on algebraic cobordism, Levine--Pandharipande~\cite{LP09} conjectured that cohomological GW cycles always lie in the tautological ring of $\Mbar_{g,n}$. The theorem above verifies this conjecture for varieties that admit degenerations into elementary pieces for which the conjecture can be checked directly, for example via localization~\cite{LP09}. In particular, it applies to complete intersections in products of projective spaces. See~\cite{MR25} and references therein for further details.

\subsection{GW cycles in Chow}

GW cycles in Chow are extremely rich, but also difficult to control. In cohomology, even before the result above, there were a number of results about GW cycles being tautological, including for complete intersections in $\mathbb P^r$, toric varieties and $G/P$'s, and curves~\cite{ABPZ,GP99,Jan17}. In Chow, there were very few examples, almost all proved via localization~\cite{GP99}. This is both because the methods are less effective and because the classes in Chow are often \underline{not} tautological. 

The flexibility of working with snc degenerations, rather than only the double point degenerations of~\cite{Li01,Li02}, becomes essential. To explain briefly, recall that the degeneration formula typically involves diagonal classes 
\[
\Delta_W \subset W\times W,
\]
where $W$ ranges over the strata of $\mathcal X_0$. To express the GW cycles of the general fiber in terms of those of the strata of the special fiber, one must decompose
\[
[\Delta_W]\in \mathsf{CH}^\star(W\times W)
\]
or some strict transform thereof. Such decompositions do not always exist. A well-behaved class of varieties in this respect is given by the \underline{linear varieties}; see~\cite{Tot14}. There are many snc (but not double point) degenerations whose strata are linear.

Using these more flexible degenerations, we prove the following. If $X\subset \mathbb P$ is a subvariety, we call a Chow cohomology class on $X$ \underline{ambient} if it is the restriction of a class from $\mathbb P$. 

\begin{theorem}
Let $X\subset \mathbb P$ be a smooth complete intersection in a product of projective spaces. The Gromov--Witten cycles of $X$ in Chow, with ambient insertions, lie in the tautological subring of $\mathsf{CH}^\star(\Mbar_{g,n})$. 
\end{theorem}

In particular, this implies the theorem on quintic threefolds stated in the introduction. 

\bibliographystyle{siam}
\bibliography{NotesonLogGWTheory}

\begin{thebibliography}{10}

\bibitem{ACW}
{\sc D.~Abramovich, C.~Cadman, and J.~Wise}, {\em Relative and orbifold
  {G}romov-{W}itten invariants}, Algebr. Geom., 4 (2017), pp.~472--500.

\bibitem{AC11}
{\sc D.~Abramovich and Q.~Chen}, {\em Stable logarithmic maps to
  {D}eligne-{F}altings pairs {II}}, Asian J. Math., 18 (2014), pp.~465--488.

\bibitem{ACGHOSS}
{\sc D.~Abramovich, Q.~Chen, D.~Gillam, Y.~Huang, M.~Olsson, M.~Satriano, and
  S.~Sun}, {\em Logarithmic geometry and moduli}, in Handbook of moduli,
  G.~Farkas and I.~Morrison, eds., 2013.

\bibitem{ACGS15}
{\sc D.~Abramovich, Q.~Chen, M.~Gross, and B.~Siebert}, {\em Decomposition of
  degenerate {G}romov-{W}itten invariants}, Compos. Math., 156 (2020),
  pp.~2020--2075.

\bibitem{ACGS25}
\leavevmode\vrule height 2pt depth -1.6pt width 23pt, {\em Punctured
  logarithmic maps}, vol.~15 of Memoirs of the European Mathematical Society,
  EMS Press, Berlin, 2025.

\bibitem{ACMUW}
{\sc D.~Abramovich, Q.~Chen, S.~Marcus, M.~Ulirsch, and J.~Wise}, {\em
  Skeletons and fans of logarithmic structures}, in Nonarchimedean and Tropical
  Geometry, M.~Baker and S.~Payne, eds., Simons Symposia, Springer, 2016,
  pp.~287--336.

\bibitem{ACMW}
{\sc D.~Abramovich, Q.~Chen, S.~Marcus, and J.~Wise}, {\em Boundedness of the
  space of stable logarithmic maps}, {J. Eur. Math. Soc.}, 19 (2017),
  pp.~2783--2809.

\bibitem{AF11}
{\sc D.~Abramovich and B.~Fantechi}, {\em Orbifold techniques in degeneration
  formulas}, Ann. Sc. Norm. Super. Pisa Cl. Sci. (5), 16 (2016), pp.~519--579.

\bibitem{AK00}
{\sc D.~Abramovich and K.~Karu}, {\em Weak semistable reduction in
  characteristic 0}, Invent. Math., 139 (2000), pp.~241--273.

\bibitem{AMW12}
{\sc D.~Abramovich, S.~Marcus, and J.~Wise}, {\em {Comparison theorems for
  Gromov--Witten invariants of smooth pairs and of degenerations}}, in Ann.
  Inst. Fourier, vol.~64, 2014, pp.~1611--1667.

\bibitem{AW}
{\sc D.~Abramovich and J.~Wise}, {\em {Birational invariance in logarithmic
  Gromov--Witten theory}}, Comp. Math., 154 (2018), pp.~595--620.

\bibitem{AP21}
{\sc A.~Abreu and M.~Pacini}, {\em The resolution of the universal {A}bel map
  via tropical geometry and applications}, Adv. Math., 378 (2021), pp.~Paper
  No. 107520, 62.

\bibitem{ABPZ}
{\sc H.~Arg\"{u}z, P.~Bousseau, R.~Pandharipande, and D.~Zvonkine}, {\em
  Gromov-{W}itten theory of complete intersections via nodal invariants}, J.
  Topol., 16 (2023), pp.~264--343.

\bibitem{AM14}
{\sc K.~B. Ascher and S.~Molcho}, {\em Logarithmic stable toric varieties and
  their moduli}, Algebraic Geometry, 3 (2016), pp.~296--319.

\bibitem{BN22}
{\sc L.~J. Barrott and N.~Nabijou}, {\em Tangent curves to degenerating
  hypersurfaces}, J. Reine Angew. Math., 793 (2022), pp.~185--224.

\bibitem{BCM20}
{\sc L.~Battistella, F.~Carocci, and C.~Manolache}, {\em Virtual classes for
  the working mathematician}, SIGMA Symmetry Integrability Geom. Methods Appl.,
  16 (2020), pp.~Paper No. 026, 38.

\bibitem{BNR22}
{\sc L.~Battistella, N.~Nabijou, and D.~Ranganathan}, {\em Gromov-{W}itten
  theory via roots and logarithms}, Geom. Topol., 28 (2024), pp.~3309--3355.

\bibitem{BNR24}
\leavevmode\vrule height 2pt depth -1.6pt width 23pt, {\em Logarithmic negative
  tangency and root stacks}, arXiv preprint arXiv:2402.08014,  (2024).

\bibitem{Bou17}
{\sc P.~Bousseau}, {\em Tropical refined curve counting from higher genera and
  lambda classes}, Invent. Math., 215 (2019), pp.~1--79.

\bibitem{CCUW}
{\sc R.~Cavalieri, M.~Chan, M.~Ulirsch, and J.~Wise}, {\em A moduli stack of
  tropical curves}, {Forum Math. Sigma}, 8 (2020), pp.~1--93.

\bibitem{CMRbook}
{\sc R.~Cavalieri, H.~Markwig, and D.~Ranganathan}, {\em Tropical and
  logarithmic methods in enumerative geometry}, vol.~52 of Oberwolfach
  Seminars, Birkh\"auser/Springer, Cham, [2023] \copyright 2023.

\bibitem{CheDegForm}
{\sc Q.~Chen}, {\em The degeneration formula for logarithmic expanded
  degenerations}, J. Algebr. Geom., 23 (2014), pp.~341--392.

\bibitem{Che10}
\leavevmode\vrule height 2pt depth -1.6pt width 23pt, {\em Stable logarithmic
  maps to {D}eligne-{F}altings pairs {I}}, Ann. of Math., 180 (2014),
  pp.~341--392.

\bibitem{CJR}
{\sc Q.~Chen, F.~Janda, and Y.~Ruan}, {\em The logarithmic gauged linear sigma
  model}, Invent. Math., 225 (2021), pp.~1077--1154.

\bibitem{CS12}
{\sc Q.~Chen and M.~Satriano}, {\em Chow quotients of toric varieties as moduli
  of stable log maps}, Algebra and Number Theory Journal, 7 (2013),
  pp.~2313--2329.

\bibitem{CFK}
{\sc I.~Ciocan-Fontanine and B.~Kim}, {\em Moduli stacks of stable toric
  quasimaps}, Adv. Math., 225 (2010), pp.~3022--3051.

\bibitem{FFR21}
{\sc S.~Felten, M.~Filip, and H.~Ruddat}, {\em Smoothing toroidal crossing
  spaces}, Forum Math. Pi, 9 (2021), pp.~Paper No. e7, 36.

\bibitem{Ful98}
{\sc W.~Fulton}, {\em Intersection theory}, vol.~2 of Ergebnisse der Mathematik
  und ihrer Grenzgebiete. 3. Folge. A Series of Modern Surveys in Mathematics
  [Results in Mathematics and Related Areas. 3rd Series. A Series of Modern
  Surveys in Mathematics], Springer-Verlag, Berlin, second~ed., 1998.

\bibitem{FP96}
{\sc W.~Fulton and R.~Pandharipande}, {\em Notes on stable maps and quantum
  cohomology}, in Algebraic geometry---{S}anta {C}ruz 1995, vol.~62, Part 2 of
  Proc. Sympos. Pure Math., Amer. Math. Soc., Providence, RI, 1997, pp.~45--96.

\bibitem{Gat02}
{\sc A.~Gathmann}, {\em {Absolute and relative Gromov-Witten invariants of very
  ample hypersurfaces}}, Duke Math. J., 115 (2002), pp.~171--203.

\bibitem{GKZ}
{\sc I.~M. Gelfand, M.~M. Kapranov, and A.~V. Zelevinsky}, {\em Discriminants,
  resultants, and multidimensional determinants}, Mathematics: Theory \&
  Applications, Birkh\"auser Boston, Inc., Boston, MA, 1994.

\bibitem{Gie84}
{\sc D.~Gieseker}, {\em A degeneration of the moduli space of stable bundles},
  J. Differential Geom., 19 (1984), pp.~173--206.

\bibitem{GP99}
{\sc T.~Graber and R.~Pandharipande}, {\em Localization of virtual classes},
  Invent. Math., 135 (1999), pp.~487--518.

\bibitem{GS13}
{\sc M.~Gross and B.~Siebert}, {\em {Logarithmic Gromov-Witten invariants}}, J.
  Amer. Math. Soc., 26 (2013), pp.~451--510.

\bibitem{GS26}
\leavevmode\vrule height 2pt depth -1.6pt width 23pt, {\em Intrinsic mirror
  symmetry}, J. Amer. Math. Soc., 39 (2026), pp.~313--451.

\bibitem{GJR}
{\sc S.~Guo, F.~Janda, and Y.~Ruan}, {\em Structure of higher genus
  {G}romov-{W}itten invariants of quintic 3-folds}, arXiv:1812.11908,  (2018).

\bibitem{Herr}
{\sc L.~Herr}, {\em {The Log Product Formula}}, arXiv:1908.04936,  (2019).

\bibitem{Hol17}
{\sc D.~Holmes}, {\em Extending the double ramification cycle by resolving the
  {A}bel-{J}acobi map}, J. Inst. Math. Jussieu, 20 (2021), pp.~331--359.

\bibitem{HMPPS}
{\sc D.~Holmes, S.~Molcho, R.~Pandharipande, A.~Pixton, and J.~Schmitt}, {\em
  Logarithmic double ramification cycles}, Invent. Math., 240 (2025),
  pp.~35--121.

\bibitem{IP01}
{\sc E.-N. Ionel and T.~H. Parker}, {\em Relative {G}romov-{W}itten
  invariants}, Ann. of Math. (2), 157 (2003), pp.~45--96.

\bibitem{Jan17}
{\sc F.~Janda}, {\em Gromov-{W}itten theory of target curves and the
  tautological ring}, Michigan Math. J., 66 (2017), pp.~683--698.

\bibitem{JPPZ}
{\sc F.~Janda, R.~Pandharipande, A.~Pixton, and D.~Zvonkine}, {\em Double
  ramification cycles on the moduli spaces of curves}, Publ. Math. IH{\'E}S,
  125 (2017), pp.~221--266.

\bibitem{KSZ91}
{\sc M.~Kapranov, B.~Sturmfels, and A.~Zelevinsky}, {\em Quotients of toric
  varieties}, Math. Ann., 290 (1991), pp.~643--655.

\bibitem{Kap93}
{\sc M.~M. Kapranov}, {\em Chow quotients of {G}rassmannians. {I}}, in I. M.
  Gelfand Seminar, vol.~16, Part 2 of Adv. Soviet Math., Amer. Math. Soc.,
  Providence, RI, 1993, pp.~29--110.

\bibitem{KKMSD}
{\sc G.~Kempf, F.~Knudsen, D.~Mumford, and B.~Saint-Donat}, {\em Toroidal
  embeddings {I}}, Lecture Notes in Mathematics, 339 (1973).

\bibitem{KH23}
{\sc P.~Kennedy-Hunt}, {\em The logarithmic quot space: foundations and
  tropicalisation}, arXiv preprint arXiv:2308.14470,  (2023).

\bibitem{KH21}
\leavevmode\vrule height 2pt depth -1.6pt width 23pt, {\em Logarithmic
  {P}andharipande-{T}homas spaces and the secondary polytope}, Trans. Amer.
  Math. Soc., 378 (2025), pp.~1--44.

\bibitem{KHSUK}
{\sc P.~Kennedy-Hunt, Q.~Shafi, and A.~U. Kumaran}, {\em Tropical refined curve
  counting with descendants}, Comm. Math. Phys., 405 (2024), pp.~Paper No. 240,
  41.

\bibitem{Kim08}
{\sc B.~Kim}, {\em Logarithmic stable maps}, in {New developments in algebraic
  geometry, integrable systems and mirror symmetry}, Adv. Stud. Pure Math.,,
  Math. Soc. Japan, Tokyo, 2010, pp.~167--200.

\bibitem{KLR23}
{\sc B.~Kim, H.~Lho, and H.~Ruddat}, {\em The degeneration formula for stable
  log maps}, Manuscripta Math., 170 (2023), pp.~63--107.

\bibitem{LP09}
{\sc M.~Levine and R.~Pandharipande}, {\em Algebraic cobordism revisited},
  Invent. Math., 176 (2009), pp.~63--130.

\bibitem{LR01}
{\sc A.-M. Li and Y.~Ruan}, {\em Symplectic surgery and {G}romov-{W}itten
  invariants of {C}alabi-{Y}au 3-folds}, Invent. Math., 145 (2001),
  pp.~151--218.

\bibitem{Li01}
{\sc J.~Li}, {\em Stable morphisms to singular schemes and relative stable
  morphisms}, J. Diff. Geom., 57 (2001), pp.~509--578.

\bibitem{Li02}
\leavevmode\vrule height 2pt depth -1.6pt width 23pt, {\em {A degeneration
  formula of GW-invariants}}, J. Diff. Geom., 60 (2002), pp.~199--293.

\bibitem{MR16}
{\sc T.~Mandel and H.~Ruddat}, {\em Descendant log {G}romov-{W}itten invariants
  for toric varieties and tropical curves}, Trans. Amer. Math. Soc., 373
  (2020), pp.~1109--1152.

\bibitem{Mano12}
{\sc C.~Manolache}, {\em Virtual pull-backs}, J. Algebr. Geom., 21 (2012),
  pp.~201--245.

\bibitem{MW17}
{\sc S.~Marcus and J.~Wise}, {\em Logarithmic compactification of the
  {A}bel-{J}acobi section}, Proc. Lond. Math. Soc. (3), 121 (2020),
  pp.~1207--1250.

\bibitem{MP06}
{\sc D.~Maulik and R.~Pandharipande}, {\em A topological view of
  {G}romov--{W}itten theory}, Topology, 45 (2006), pp.~887--918.

\bibitem{MR26}
{\sc D.~Maulik and D.~Ranganathan}, {\em {The GW/PT conjectures for toric
  pairs}}, In preparation.

\bibitem{MR20}
\leavevmode\vrule height 2pt depth -1.6pt width 23pt, {\em Logarithmic
  {D}onaldson--{T}homas theory}, Forum Math. Pi, 12 (2024), p.~Paper No. e9.

\bibitem{MR25}
\leavevmode\vrule height 2pt depth -1.6pt width 23pt, {\em {Gromov-Witten
  theory, degenerations, and the tautological ring}}, arXiv:2510.04779,
  (2025).

\bibitem{MR23}
\leavevmode\vrule height 2pt depth -1.6pt width 23pt, {\em Logarithmic
  enumerative geometry for curves and sheaves}, Camb. J. Math., 13 (2025),
  pp.~51--172.

\bibitem{Mi03}
{\sc G.~Mikhalkin}, {\em Enumerative tropical geometry in {${\mathbb{R}^2}$}},
  J. Amer. Math. Soc, 18 (2005), pp.~313--377.

\bibitem{SCM24}
{\sc S.~C. Mok}, {\em {Logarithmic Fulton-Macpherson compactification of
  configuration spaces}}, Preprint,  (2024).

\bibitem{Mol16}
{\sc S.~Molcho}, {\em Universal stacky semistable reduction}, Israel J. Math.,
  242 (2021), pp.~55--82.

\bibitem{MR21}
{\sc S.~Molcho and D.~Ranganathan}, {\em A case study of intersections on
  blowups of the moduli of curves}, Algebra Number Theory, 18 (2024),
  pp.~1767--1816.

\bibitem{MW18}
{\sc S.~Molcho and J.~Wise}, {\em The logarithmic {P}icard group and its
  tropicalization}, Compos. Math., 158 (2022), pp.~1477--1562.

\bibitem{Nab19}
{\sc N.~Nabijou}, {\em {Recursion Formulae in Logarithmic Gromov-Witten Theory
  and Quasimap Theory}}, PhD thesis, Imperial College London, 2019.

\bibitem{NR19}
{\sc N.~Nabijou and D.~Ranganathan}, {\em Gromov-{W}itten theory with maximal
  contacts}, Forum Math. Sigma, 10 (2022), pp.~Paper No. e5, 34.

\bibitem{NS06}
{\sc T.~Nishinou and B.~Siebert}, {\em Toric degenerations of toric varieties
  and tropical curves}, Duke Math. J., 135 (2006), pp.~1--51.

\bibitem{Par11}
{\sc B.~Parker}, {\em {Gromov Witten invariants of exploded manifolds}},
  arXiv:1102.0158,  (2011).

\bibitem{Par17}
\leavevmode\vrule height 2pt depth -1.6pt width 23pt, {\em Three dimensional
  tropical correspondence formula}, Commun. Math. Phys., 353 (2017),
  pp.~791--819.

\bibitem{Par17a}
\leavevmode\vrule height 2pt depth -1.6pt width 23pt, {\em {Tropical gluing
  formulae for Gromov-Witten invariants}}, arXiv:1703.05433,  (2017).

\bibitem{Par19}
\leavevmode\vrule height 2pt depth -1.6pt width 23pt, {\em Holomorphic curves
  in exploded manifolds: virtual fundamental class}, Geom. Topol., 23 (2019),
  pp.~1877--1960.

\bibitem{R15b}
{\sc D.~Ranganathan}, {\em {Skeletons of stable maps I: rational curves in
  toric varieties}}, J. Lond. Math. Soc., 95 (2017), pp.~804--832.

\bibitem{R19}
\leavevmode\vrule height 2pt depth -1.6pt width 23pt, {\em Logarithmic
  {G}romov-{W}itten theory with expansions}, Algebr. Geom., 9 (2022),
  pp.~714--761.

\bibitem{RUK22}
{\sc D.~Ranganathan and A.~Urundolil~Kumaran}, {\em Logarithmic
  {G}romov-{W}itten theory and double ramification cycles}, J. Reine Angew.
  Math., 809 (2024), pp.~1--40.

\bibitem{RW19}
{\sc D.~Ranganathan and J.~Wise}, {\em Rational curves in the logarithmic
  multiplicative group}, Proc. Amer. Math. Soc., 148 (2020), pp.~103--110.

\bibitem{Tev07}
{\sc J.~Tevelev}, {\em Compactifications of subvarieties of tori}, Amer. J.
  Math., 129 (2007), pp.~1087--1104.

\bibitem{Tot14}
{\sc B.~Totaro}, {\em Chow groups, {C}how cohomology, and linear varieties},
  Forum Math. Sigma, 2 (2014), pp.~Paper No. e17, 25.

\bibitem{TY23}
{\sc H.-H. Tseng and F.~You}, {\em A {G}romov-{W}itten theory for simple
  normal-crossing pairs without log geometry}, Comm. Math. Phys., 401 (2023),
  pp.~803--839.

\bibitem{U15}
{\sc M.~Ulirsch}, {\em Tropical compactification in log-regular varieties},
  Math. Z., 280 (2015), pp.~195--210.

\bibitem{Vak00}
{\sc R.~Vakil}, {\em {The enumerative geometry of rational and elliptic curves
  in projective space.}}, J. Reine Angew. Math., 529 (2000), pp.~101--153.

\bibitem{VGNS}
{\sc M.~Van~Garrel, N.~Nabijou, and Y.~Schuler}, {\em Gromov--witten theory of
  bicyclic pairs}, Trans. Amer. Math. Soc.,  (2026).

\bibitem{Wis16a}
{\sc J.~Wise}, {\em Moduli of morphisms of logarithmic schemes}, {Algebra
  Number Theory}, 10 (2016), pp.~695--735.

\bibitem{Wis16b}
\leavevmode\vrule height 2pt depth -1.6pt width 23pt, {\em Uniqueness of
  minimal morphisms of logarithmic schemes}, Algebr. Geom., 6 (2019),
  pp.~50--63.

\end{thebibliography}

\end{document}